\documentclass[12pt]{article}
\usepackage{amsfonts, amsmath, amssymb, latexsym, eucal, amscd} 

\usepackage{amsthm}
\usepackage{amscd}

\usepackage{cite}

\newtheorem{definition}{Definition}[section]
\newtheorem{theorem}[definition]{Theorem}
\newtheorem{lemma}[definition]{Lemma}
\newtheorem{corollary}[definition]{Corollary}
\newtheorem{remark}[definition]{Remark}
\newtheorem{example}[definition]{Example}

\newtheorem{problem}[definition]{Problem}
\newtheorem{note}[definition]{Note}

\newtheorem{proposition}[definition]{Proposition}

\begin{document} 

\title{\bf The nucleus of a $Q$-polynomial 
distance-regular graph
}
\author{
Paul Terwilliger}
\date{}

\maketitle
\begin{abstract} 
Let $\Gamma$ denote a  $Q$-polynomial distance-regular graph with diameter $D\geq 1$. For a vertex $x$ of $\Gamma$
the corresponding subconstituent algebra $T=T(x)$ is generated by the adjacency matrix $A$ of $\Gamma$ and the
dual adjacency matrix $A^*=A^*(x)$ of $\Gamma$ with respect to $x$.
We introduce a $T$-module $\mathcal N = \mathcal N(x)$ called the nucleus of $\Gamma$ with respect to $x$.
We  describe $\mathcal N$ from various points of view. We show  that all the irreducible $T$-submodules of $\mathcal N$
are thin.  Under the assumption that $\Gamma$ is a nonbipartite dual polar graph,
we give an explicit basis for $\mathcal N$ and the action of $A, A^*$ on this basis. The basis is in bijection with
the set of elements for the projective geometry $L_D(q)$, where $GF(q)$ is the finite field used to define 
$\Gamma$.

\medskip

\noindent
{\bf Keywords}. Dual polar graph; projective geometry; $Q$-polynomial property; subconstituent algebra.
\hfil\break
\noindent {\bf 2020 Mathematics Subject Classification}.
Primary: 05E30. Secondary: 05C50, 06C05.
 \end{abstract}
 
\section{Introduction} This paper is about the subconstituent algebras of a $Q$-polynomial distance-regular graph. Surveys about this general topic
can be found in \cite[Chapter 6]{Nbbit}, \cite{Ndkt,Nint}. 
\medskip

\noindent Before we explain our purpose, let us briefly recall the subconstituent algebra concept (formal definitions start  in Section 2).
Let $\Gamma$ denote a  distance-regular graph, with vertex set $X$, path-length distance function $\partial$, and diameter $D\geq 1$.
Let $A \in {\rm Mat}_X(\mathbb C)$ denote the adjacency matrix of $\Gamma$.
 We assume that $\Gamma$ is $Q$-polynomial with respect to the ordering $\lbrace E_i\rbrace_{i=0}^D$ of the primitive idempotents of $A$.
 Fix a vertex $x \in X$. For $0 \leq i \leq D$ define the set $\Gamma_i(x)= \lbrace y \in X \vert \partial(x,y)=i\rbrace$. This set is often
 called the $i$th subconstituent of $\Gamma$ with respect to $x$. The corresponding dual primitive idempotent is the diagonal matrix $E^*_i =E^*_i(x)$ in ${\rm Mat}_X(\mathbb C)$
that has the following entries. For $y \in X$ the  $(y,y)$-entry is $1$ if $y \in \Gamma_i(x)$, and $0$ if $y \not\in \Gamma_i(x)$.
The subconstituent algebra $T=T(x)$ is the subalgebra of ${\rm Mat}_X(\mathbb C)$ generated by $A$ and 
 a certain diagonal matrix $A^*=A^*(x)$ called the dual adjacency matrix of $\Gamma$ with respect to $x$.
The matrices $\lbrace E^*_i\rbrace_{i=0}^D$ are the primitive idempotents of $A^*$.
 By a $T$-module, we mean a $T$-submodule of the
standard module $V = \mathbb C^X$.
\medskip

\noindent In this paper, we introduce a $T$-module $\mathcal N = \mathcal N(x)$ called the nucleus of $\Gamma$ with respect to $x$.
We  describe $\mathcal N$ from various points of view. We show  that all the irreducible $T$-submodules of $\mathcal N$
have a property called thin.  Under the assumption that $\Gamma$ is a nonbipartite dual polar graph,
we give an explicit basis for $\mathcal N$ and the action of $A, A^*$ on this basis.
\medskip

\noindent  We will describe $\mathcal N$ in more detail, after a few definitions.
Let $W$ denote an irreducible $T$-module. By \cite[Corollary~5.7]{NsomeAlg} we have ${\rm dim}\, E^*_iW \leq 1$ $(0 \leq i \leq D)$ if and only if 
${\rm dim}\, E_iW \leq 1$ $(0 \leq i \leq D)$; whenever this happens we say that $W$ is thin.
By the  endpoint of $W$ we mean
\begin{align*}
{\rm min}\lbrace i \vert 0 \leq i \leq D, \; E^*_iW \not=0 \rbrace.
\end{align*}
\noindent By the diameter of $W$, we mean
\begin{align*}
\bigl \vert \lbrace i \vert 0 \leq i \leq D, \; E^*_iW \not=0 \rbrace  \bigr\vert -1.
\end{align*}
By the dual endpoint of $W$ we mean
\begin{align*}
{\rm min}\lbrace i \vert 0 \leq i \leq D, \; E_iW \not=0 \rbrace.
\end{align*}
\noindent By \cite[Corollary~3.3]{pasc}, the diameter of $W$ is equal to
\begin{align*}
\bigl \vert \lbrace i \vert 0 \leq i \leq D, \; E_iW \not=0 \rbrace  \bigr\vert -1.
\end{align*}
\noindent By the displacement of $W$ we mean  $r+t-D+d$, where $r$ (resp. $t$) (resp. $d$)
denotes the endpoint (resp. dual endpoint) (resp. diameter) of $W$. The displacement of $W$ is nonnegative \cite[Lemma~4.2]{ds}.
We show that the following are equivalent: 
\begin{enumerate}
\item[\rm (i)] the displacement of $W$ is zero; 
\item[\rm (ii)] 
 $r=t$ and $d=D-2r$.
 \end{enumerate}
 \noindent 
  We show that if (i), (ii) hold then $W$ is thin, and in this case $r$ determines $W$ up to isomorphism of $T$-modules. 
  \medskip
  
 \noindent 
We introduce a $T$-module called the nucleus of $\Gamma$ with respect to $x$. By definition, this $T$-module is the span of the irreducible $T$-modules that have displacement 0.
\medskip

\noindent We describe the nucleus as follows.
 For $0 \leq i \leq D$ define a subspace $\mathcal N_i = \mathcal N_i(x)$ of $V$ by
\begin{align*}
\mathcal N_i = (E^*_0V+E^*_1V+\cdots + E^*_iV) \cap (E_0V+E_1V+\cdots + E_{D-i}V).
\end{align*}
By \cite[Corollary~5.8, Theorem~6.2]{ds} the sum $\sum_{i=0}^D \mathcal N_i$ is direct; denote the sum by $\mathcal N = \mathcal N(x)$.
We show that the following are the same:
\begin{enumerate}
\item [\rm (i)] the subspace $\mathcal N=\mathcal N(x)$;
\item[\rm (ii)] the nucleus of $\Gamma$ with respect to $x$.
\end{enumerate}
\noindent  Next, we assume that $\Gamma$ is a nonbipartite dual polar graph  \cite[p.~352]{Nbbit}, \cite[p.~303]{Nbannai}, \cite[Example~6.1(6)]{tSub3}.
Using the vertex $x$, we define a binary relation $\sim$ on $X$ as follows. 
 For $y, z \in X$ we declare $y \sim z$
whenever both
\begin{enumerate}
\item[\rm (i)] $\partial(x,y) = \partial(x,z)$;
\item[\rm (ii)]  $y,z$ are in the same connected component of $\Gamma_i(x)$, where $i=\partial(x,y)=\partial(x,z)$.
\end{enumerate}
Note that $\sim$ is an equivalence relation.
We show that the nucleus $\mathcal N= \mathcal N(x)$ has a basis consisting of the
characteristic vectors of the $\sim$ equivalence classes.  We give the actions of $A, A^*$ on this basis.
\medskip

\noindent 
By the definition of a dual polar graph, the  vertices of $\Gamma$ are certain $D$-dimensional subspaces of a vector space
over a  finite field $GF(q)$. In particular, the vertex $x$ is a $D$-dimensional vector space over $GF(q)$.
 Let the set $\mathcal P$ consist of the subspaces of $x$. The set $\mathcal P$, together with the inclusion
relation, is a poset called the projective geometry $L_D(q)$ \cite[p.~439]{Nbcn}, \cite[Chapter~1]{Ncameron}. We give a bijection from $\mathcal P$
to the above basis for $\mathcal N$. Via the bijection, the action of $A$ on $\mathcal N$ becomes a weighted adjacency map for
$L_D(q)$ in the sense of \cite[Section~1]{LNq2}. This weighted adjacency map  for $L_D(q)$ is a variation on the one given in \cite[Definition~7.1]{LNq2}.
We would like to acknowledge that a similar weighted adjacency map for $L_D(q)$ showed up earlier in the work of 
Bernard, Cramp{\'e}, and Vinet
\cite[Theorem~7.1]{PAB} concerning the dual polar graph with symplectic type and $q$ a prime.
\medskip

\noindent 
We just mentioned a bijection from $\mathcal P$
to a basis for $\mathcal N$. The bijection shows that the cardinality of $\mathcal P$ is equal to the dimension of $\mathcal N$,
which must be a sum
of $q$-binomial coefficients:
\begin{align*}
{\rm dim}\, \mathcal N = \sum_{i=0}^D \binom{D}{i}_q.
\end{align*}

\noindent  Some our main results are applications of  the theory of tridiagonal pairs \cite{NsomeAlg}. In the main body of the paper, our first goal is to
 develop this theory and explain how it relates to $Q$-polynomial distance-regular graphs.
\medskip

\noindent The paper is organized as follows. Section 2 contains some preliminaries.
Section 3 is about tridiagonal pairs. 
In Section 4, we review the subconstituent algebra and the $Q$-polynomial property for distance-regular graphs.
In Section 5, we explain how tridiagonal pairs are related to $Q$-polynomial distance-regular graphs.
In Sections 6, 7 we introduce the nucleus and describe it in several ways.
In Section 8, we describe the nucleus for several elementary examples.
In Section 9 we recall the dual polar graphs.
In Sections 10--13, we describe the nucleus for a nonbipartite dual polar graph.
Section 14 contains some open problems concerning the nucleus.
\medskip

\noindent The main results of this paper are 
Propositions \ref{lem:Wdata}, \ref{lem:nucISO} and
Theorems \ref{thm:ND}, \ref{thm:DPN},  \ref{thm:bij}, \ref{thm:Pmain}.

\section{Preliminaries} 
We now begin our formal argument. The following concepts and notation will be used throughout the paper.
Every algebra that we discuss, is understood to be associative and have a multiplicative identity.
A subalgebra has the same multiplicative identity as the parent algebra.
Let $\mathbb F$ denote a field, and let $V$ denote a  vector space over $\mathbb F$  that has finite positive dimension.
Let ${\rm End}(V)$ denote the $\mathbb F$-algebra consisting of the $\mathbb F$-linear maps from $V$ to $V$.
For $A \in {\rm End}(V)$ let  $\langle A \rangle$ denote the subalgebra of ${\rm End}(V)$ generated by $A$.
The map $A$ is said to be {\it diagonalizable} whenever $V$ is spanned by the eigenspaces of $A$. Assume for the moment
that $A$ is diagonalizable, and let $U$ denote an eigenspace of $A$. The corresponding {\it primitive idempotent} of $A$ is the
map in ${\rm End}(V)$ that acts as the identity on $U$ and as zero on every other eigenspace of $A$. By linear algebra,
the primitive idempotents of $A$ form a basis for $\langle A \rangle$; see for example \cite[Section~2]{altDRG}.

\section{Tridiagonal pairs}

\noindent In this section, we describe some results about tridiagonal pairs. Later in the paper, we will apply these results
to $Q$-polynomial distance-regular graphs. Background information about tridiagonal pairs can be found in  \cite{itoOq, IKT, tdanduq, TDqRac, augIto, LSnotes,vidunas}. 
\medskip

\noindent Throughout this section, $V$ denotes a  vector space over $\mathbb F$  that has finite positive dimension.

\begin{definition} \label{def:TDP} \rm (See \cite[Definition~1.1]{NsomeAlg}.)
By a {\it tridiagonal pair on $V$}, we mean an ordered pair $A, A^*$ of elements in ${\rm End}(V)$ that 
satisfy the following conditions {\rm (i)--(iv)}:
\begin{enumerate}
\item[\rm (i)] each of $A, A^*$ is diagonalizable;
\item[\rm (ii)] there exists an ordering $\lbrace V_i \rbrace_{i=0}^d$ of the eigenspaces of $A$ such that
\begin{align}
          A^* V_i \subseteq V_{i-1} + V_i + V_{i+1} \qquad \quad (0 \leq i \leq d), \label{eq:sum1}
\end{align}
\noindent where $V_{-1}=0$ and $V_{d+1} =0$;
\item[\rm (iii)] there exists an ordering $\lbrace V^*_i \rbrace_{i=0}^{\delta}$ of the eigenspaces of $A^*$ such that
\begin{align}
          A V^*_i \subseteq V^*_{i-1} + V^*_i + V^*_{i+1} \qquad \quad (0 \leq i \leq \delta), \label{eq:sum2}
\end{align}
\noindent where $V^*_{-1}=0$ and $V^*_{\delta+1} =0$;
\item[\rm (iv)] there does not exist a subspace $W \subseteq V$ such that $A W \subseteq W$, $A^* W \subseteq W$, $W \not=0$, $W\not=V$.
\end{enumerate}
\end{definition}

\begin{note} \rm According to a common notational convention, $A^*$ denotes the conjugate-transpose of $A$. We are not following this convention. For a tridiagonal
pair $A, A^*$ the $\mathbb F$-linear maps $A$ and $A^*$ are arbitrary subject to (i)--(iv) above.
\end{note}

\noindent Let $A, A^*$ denote a tridiagonal pair on $V$, as in Definition \ref{def:TDP}. By \cite[Lemma~4.5]{NsomeAlg} the integers $d$ and $\delta$ from \eqref{eq:sum1} and \eqref{eq:sum2}
are equal; we call this common value the {\it diameter} of the  pair. An ordering $\lbrace V_i \rbrace_{i=0}^d$ of the eigenspaces of $A$ is called {\it standard}
whenever it satisfies \eqref{eq:sum1}. We comment on the uniqueness of the standard ordering. Assume that the ordering $\lbrace V_i \rbrace_{i=0}^d$
is standard. By \cite[Lemma~2.4]{NsomeAlg}, the inverted ordering $\lbrace V_{d-i}\rbrace_{i=0}^d$ is standard and no further ordering is standard. An ordering of the primitive idempotents
of $A$ is called {\it standard} whenever the corresponding ordering of the eigenspaces of $A$ is standard. A similar discussion applies to $A^*$.
\medskip

\noindent  As we investigate a tridiagonal pair, it is convenient to use the concept of a tridiagonal system.  This concept is defined as follows.

\begin{definition}\label{def:TDS} \rm (See \cite[Definition~2.1]{NsomeAlg}, \cite[Definition~2.1]{sharp}.) By a {\it tridiagonal system on $V$}, we mean a sequence
\begin{align*}
\Phi = (A; \lbrace E_i \rbrace_{i=0}^d; A^*; \lbrace E^*_i \rbrace_{i=0}^d )
\end{align*}
that satisfies the following conditions (i)--(iii):
\begin{enumerate}
\item[\rm (i)]  $A, A^*$ is a tridiagonal pair on $V$;
\item[\rm (ii)]  $\lbrace E_i \rbrace_{i=0}^d$ is a standard ordering of the primitive idempotents of $A$;
\item[\rm (iii)] $\lbrace E^*_i \rbrace_{i=0}^d$ is a standard ordering of the primitive idempotents of $A^*$.
\end{enumerate}
We call $d$ the {\it diameter of $\Phi$}.
\end{definition}

\noindent  Next, we give some basic facts about tridiagonal systems.
Let $\Phi = (A; \lbrace E_i \rbrace_{i=0}^d; A^*; \lbrace E^*_i \rbrace_{i=0}^d )$ denote a tridiagonal system on $V$. 
Then the following are tridiagonal systems on $V$:
\begin{align*}
\Phi^\downarrow &=  (A; \lbrace E_i \rbrace_{i=0}^d; A^*; \lbrace E^*_{d-i} \rbrace_{i=0}^d ); \\
\Phi^\Downarrow &=  (A; \lbrace E_{d-i} \rbrace_{i=0}^d; A^*; \lbrace E^*_i \rbrace_{i=0}^d ); \\
\Phi^* &=  (A^*; \lbrace E^*_i \rbrace_{i=0}^d; A; \lbrace E_i \rbrace_{i=0}^d ).
\end{align*}

\noindent 
Let $\Phi = (A; \lbrace E_i \rbrace_{i=0}^d; A^*; \lbrace E^*_i \rbrace_{i=0}^d )$ denote a tridiagonal system on $V$. 
By \cite[Lemma~2.4]{NsomeAlg} and \cite[Lemma~2.2]{sharp},  the following hold for $0 \leq i,j\leq d$:
\begin{align*}
E^*_i A E^*_j &= \begin{cases} 0, & \mbox{if $\vert i-j\vert >1$}; \\
                                            \not=0, & \mbox{if $\vert i-j\vert=1$};
                         \end{cases} \\
E_i A^* E_j &= \begin{cases} 0, & \mbox{if $\vert i-j\vert >1$}; \\
                                            \not=0, & \mbox{if $\vert i-j\vert=1$}.
                         \end{cases} 
\end{align*}

\noindent 
By \cite[Corollary~5.7]{NsomeAlg}, for $0 \leq i \leq d$ the following subspaces have the same dimension:
\begin{align*}
E_iV, \qquad E_{d-i}V, \qquad E^*_iV, \qquad E^*_{d-i}V.
\end{align*}
\noindent This common dimension is denoted by $\rho_i$. By construction, $\rho_i \not=0$ $(0 \leq i \leq d)$. Also by construction, $\rho_i = \rho_{d-i}$ $(0 \leq i \leq d)$.
By \cite[Corollary~6.6]{NsomeAlg} we have $\rho_{i-1} \leq \rho_i$ for $1 \leq i \leq d/2$. The sequence $\lbrace \rho_i \rbrace_{i=0}^d$ is called the {\it shape} of $\Phi$.
We call $\Phi$ a {\it Leonard system} whenever $\rho_i=1$ for $0\leq i \leq d$; see \cite[Definition~1.4]{LS99}.
\medskip

\noindent By \cite[Definition~5.1, Theorem~6.7]{switch}, if $\Phi$ is a Leonard system then there exists an element $S \in \langle A \rangle $ such that $S E^*_0V=E^*_dV$.
A bit later, it was shown \cite[Theorem~5.2]{nomSplit} that the following are equivalent:
\begin{enumerate}
\item[\rm (i)] there exists a nonzero $S \in \langle A \rangle$ such that $S E^*_0V \subseteq E^*_dV$;
\item[\rm (ii)] $\Phi$ is a Leonard system.
\end{enumerate}

\noindent The previous result will play an important role later in the paper. Because of its importance and in order to illuminate the ideas involved, we will give a short proof. We will also strengthen the statement, as follows.

\begin{proposition}\label{thm:main1} Let $\Phi = (A; \lbrace E_i \rbrace_{i=0}^d; A^*; \lbrace E^*_i \rbrace_{i=0}^d )$ denote a tridiagonal system on $V$.
Then the following are equivalent:
\begin{enumerate}
\item[\rm (i)] there exists a nonzero $v \in E^*_0V$ and  nonzero $S \in \langle A \rangle$ such that $Sv \in E^*_dV$;
\item[\rm (ii)] $\Phi$ is a Leonard system.
\end{enumerate}
\noindent Assume that {\rm (i), (ii)} hold. Then $S E^*_0V=E^*_dV$.
\end{proposition}
\begin{remark} \rm The tridiagonal system $\Phi$ is called  {\it sharp} whenever $\rho_0=1$, see \cite[Definition~1.5]{sharp}. If $\mathbb F$ is algebraically closed, then $\Phi$ is sharp \cite[Theorem~1.3]{structure}.
If $\Phi$ is sharp, then  \cite[Theorem~5.2]{nomSplit} and Proposition \ref{thm:main1}  are equivalent.
\end{remark}

\noindent We will prove Proposition \ref{thm:main1} after a few comments. These comments refer to the tridiagonal system $\Phi$ in Proposition \ref{thm:main1}.
For $0 \leq i \leq d$ let $\theta_i$ (resp. $\theta^*_i$) denote the eigenvalue for $A$ (resp. $A^*$) associated with
$E_i$ (resp. $E^*_i$). By construction, the scalars $\lbrace \theta_i \rbrace_{i=0}^d$ are mutually distinct and contained in $\mathbb F$. Similarly, the scalars  $\lbrace \theta^*_i \rbrace_{i=0}^d$ are mutually distinct and contained in $\mathbb F$.
For $0 \leq i \leq d$ define
\begin{align*}
U_i = (E^*_0V+E^*_1V+\cdots+ E^*_iV) \cap (E_0 V + E_1V+\cdots + E_{d-i}V).
\end{align*}
For example, $U_0=E^*_0V$ and $U_d=E_0V$.
By \cite[Theorem~4.6]{NsomeAlg}, the sum
$V = \sum_{i=0}^d U_i$ is direct.
By \cite[Corollary~5.7]{NsomeAlg}, ${\rm dim}\,U_i =\rho_i$ $(0 \leq i \leq d)$.
By \cite[Theorem~4.6]{NsomeAlg}, we have
\begin{align*}
&(A-\theta_{d-i} I ) U_i \subseteq U_{i+1} \quad (0 \leq i \leq d-1), \qquad \quad (A-\theta_0I)U_d=0, \\
&(A^*-\theta^*_iI) U_i \subseteq U_{i-1} \quad (1 \leq i \leq d), \qquad \quad (A^*-\theta^*_0 I) U_0=0.
\end{align*}
\noindent Following \cite[Section~6]{NsomeAlg}, we define a map $R \in {\rm End}(V)$ that acts on $U_i$ as $A-\theta_{d-i} I$ $(0 \leq i \leq d)$. By construction
\begin{align*}
R U_i \subseteq U_{i+1} \quad (0 \leq i \leq d-1), \qquad \quad RU_d=0.
\end{align*}
We call $R$ the {\it raising map} for $\Phi$.
By \cite[Lemma~6.5]{NsomeAlg}, the map
$R^{d-2i}: U_i \to U_{d-i}$ is a bijection for $0 \leq i \leq d/2$. In particular, the map $R^d: U_0 \to U_d$ is a bijection.
Shortly, we will say more about $R$.
\medskip

\noindent Let $\lambda$ denote an indeterminate. For $0 \leq i \leq d$ define a polynmial
\begin{align*}
\eta_i(\lambda) = (\lambda-\theta_d)(\lambda-\theta_{d-1}) \cdots (\lambda - \theta_{d-i+1}).
\end{align*}
We interpret $\eta_0(\lambda)=1$.  The polynomial $\eta_i(\lambda)$ is monic with degree $i$ $(0 \leq i \leq d)$.  Let $0 \not=v \in E^*_0V$. By construction,
\begin{align}
& \eta_i(A) v \in U_i \qquad \qquad (0 \leq i \leq d).
 \label{eq:Ri}
 \end{align}
 \noindent By \eqref{eq:Ri} and our comments about $R$,
 \begin{align}
&\eta_i(A)v =R^i v \not=0 \qquad \qquad (0 \leq i \leq d).
 \label{eq:tauBasis}
\end{align}
Recall that $\langle A \rangle$ is the subalgebra of ${\rm End}(V)$ generated by $A$.
This subalgebra has dimension $d+1$. Each of the following is a basis for $\langle A \rangle$:
\begin{align*}
\lbrace E_i \rbrace_{i=0}^d, \qquad \quad 
\lbrace A^i \rbrace_{i=0}^d, \qquad \quad
\lbrace \eta_i(A) \rbrace_{i=0}^d.
\end{align*}

 \begin{lemma} \label{lem:inj} {\rm (See \cite[Lemma~3.1]{nomSplit}.)} Let $\Phi = (A; \lbrace E_i \rbrace_{i=0}^d; A^*; \lbrace E^*_i \rbrace_{i=0}^d )$ denote a tridiagonal system on $V$.
 \begin{enumerate}
 \item[\rm (i)]
  Let $0\not=v \in E^*_0V$. Then the map $\langle A \rangle \to V, B \mapsto Bv$ is injective.
   \item[\rm (ii)]
  Let $0 \not=w \in E^*_dV$. Then the map $\langle A \rangle \to V, B \mapsto Bw$ is injective.
\end{enumerate}  
\end{lemma}
\begin{proof} (i) The elements $\lbrace \eta_i(A)\rbrace_{i=0}^d $ form a basis for $\langle A \rangle$. The vectors
 $\lbrace \eta_i(A)v\rbrace_{i=0}^d$ are linearly independent  by \eqref{eq:Ri}, \eqref{eq:tauBasis} and since the sum $\sum_{i=0}^d U_i$ is direct.
 The result follows. \\
 \noindent (ii)
 Apply (i) above  to the tridiagonal system $\Phi^\downarrow$.
\end{proof}

\noindent We are now ready to prove Proposition \ref{thm:main1}.
\medskip

\noindent {\it Proof of Proposition \ref{thm:main1}}: ${\rm (i)} \Rightarrow {\rm (ii)}$: Consider the subspace $W= \langle A \rangle v \subseteq V$.  Note that $W$ has a basis $\lbrace E_i v \rbrace_{i=0}^d$,
by Lemma \ref{lem:inj} and since $\lbrace E_i  \rbrace_{i=0}^d$ is a basis for $\langle A \rangle$.
We will show that $W=V$.
By construction, $W\not=0$ and $A W \subseteq W$. To show that $W=V$, by Definition  \ref{def:TDP}(iv) it suffices to show that $A^*W \subseteq W$. The vectors $\lbrace \eta_i(A)\rbrace_{i=0}^d$ form a basis for $\langle A \rangle$.
Write
\begin{align*}
S = \sum_{i=0}^d \alpha_i \eta_i(A) \qquad \qquad \alpha_i \in \mathbb F.
\end{align*}
 Note that $A^*v =\theta^*_0v$ since $v \in E^*_0V$. \\
We first claim that
\begin{align}
\alpha_i (A^*- \theta^*_i I) \eta_i(A)v = \alpha_{i-1} (\theta^*_d - \theta^*_{i-1}) \eta_{i-1}(A)v \qquad \quad (1 \leq i \leq d).
\label{eq:claim2}
\end{align}
To prove the claim, note that $(A^*-\theta^*_d I)Sv=0$ since $Sv \in E^*_dV$.
We have
\begin{align*}
0 &= (A^*-\theta^*_dI) Sv \\
   &=  (A^*-\theta^*_dI) \sum_{i=0}^d \alpha_i \eta_i(A)v \\
     &=   \sum_{i=0}^d \alpha_i (A^*-\theta^*_dI)  \eta_i(A)v \\
 &=   \sum_{i=0}^d \alpha_i (A^*-\theta^*_i I+ \theta^*_i I-\theta^*_dI)  \eta_i(A)v \\
 &=   \sum_{i=0}^d \alpha_i (A^*-\theta^*_i I)\eta_i(A)v-        \sum_{i=0}^d \alpha_i( \theta^*_d-\theta^*_i)  \eta_i(A)v \\
  &=   \sum_{i=1}^d \alpha_i (A^*-\theta^*_i I)\eta_i(A)v -       \sum_{i=0}^{d-1} \alpha_i( \theta^*_d-\theta^*_i)  \eta_i(A)v \\
   &=   \sum_{i=1}^d \alpha_i (A^*-\theta^*_i I)\eta_i(A)v -       \sum_{i=1}^{d} \alpha_{i-1}( \theta^*_d-\theta^*_{i-1})  \eta_{i-1}(A)v \\
   & = \sum_{i=1}^d \Bigl( \alpha_i (A^*-\theta^*_i I)\eta_i(A)v - \alpha_{i-1}(\theta^*_d-\theta^*_{i-1})\eta_{i-1}(A)v    \Bigr).
\end{align*}
In the previous sum, the $i$th summand is contained in $U_{i-1}$ $(1 \leq i \leq d)$. By this and since the sum $\sum_{j=0}^d U_j$ 
is direct,  the $i$th summand must be zero  $(1 \leq i \leq d)$.  This gives \eqref{eq:claim2}, and
the first claim is proven. \\
Our second claim is that $\alpha_i \not=0$ $(0 \leq i \leq d)$. We prove this claim by induction on $i$. First assume that $i=0$. 
To prove $\alpha_0\not=0$, we assume that $\alpha_0=0$ and get a contradiction.
Using $E_dA=\theta_dE_d$ we obtain
\begin{align*}
E_dSv = E_d \sum_{i=1}^d \alpha_i \eta_i(A)v 
            = \sum_{i=1}^d \alpha_i \eta_i(\theta_d) E_dv 
             =0,
\end{align*}
with the last equality holding because  $\eta_i(\theta_d)=0$ for $1 \leq i \leq d$. We have $Sv \in E^*_dV$ and $E_d Sv=0$ and $0 \not=E_d \in \langle A \rangle$, so
 $Sv =0$ by Lemma \ref{lem:inj}(ii). We have 
 $v \in E^*_0V$ and $Sv =0$ and $0 \not=S \in \langle A \rangle$, 
 so $v=0$ by Lemma \ref{lem:inj}(i).
This is a contradiction, so $\alpha_0 \not=0$. 
Next assume that $1 \leq i \leq d$. Then $\alpha_i \not=0$ by induction and the first claim, since the right-hand side of \eqref{eq:claim2} is nonzero.
The second claim is proven. \\
We can now easily show that $A^* W \subseteq W$. The vectors  $\lbrace \eta_i(A)v\rbrace_{i=0}^d$ form a basis for $W$.
We show that $A^* \eta_i(A)v \in W$ for $0 \leq i \leq d$.
First assume that $i=0$. Then  $A^* \eta_0(A)v = A^* v = \theta^*_0v \in W$. Next assume that $1\leq i \leq d$.
By the first and second claim,
\begin{align*}
A^* \eta_i(A)v \in {\rm Span} \Bigl \lbrace \eta_i(A)v, \eta_{i-1}(A)v\Bigr \rbrace \subseteq W.
\end{align*}
We have shown that $A^*W \subseteq W$, so $W=V$. By this and the construction, ${\dim}\,E_i V=1$ $(0 \leq i \leq d)$.
 Consequently $\Phi$ is a Leonard system.
We have proved Proposition \ref{thm:main1} in the direction ${\rm (i)} \Rightarrow {\rm (ii)}$. The reverse direction 
${\rm (ii)} \Rightarrow {\rm (i)}$ is proved in \cite[Theorem~6.7]{switch}. 
\medskip

\noindent Assume that (i), (ii) hold. By (ii), the subspaces $E^*_0V$ and $E^*_dV$ have dimension one.
By construction $0 \not=v \in E^*_0V$, so $v$ is a basis for $E^*_0V$.
By construction and Lemma \ref{lem:inj}(i), $0 \not=Sv \in E^*_dV$. Therefore $Sv$ is a basis for $E^*_dV$.
It follows that $SE^*_0V=E^*_dV$.
 \hfill $\Box$
 \medskip
 
 \noindent We mention some variations on the above results.

  \begin{corollary} \label{lem:injd} Let $\Phi = (A; \lbrace E_i \rbrace_{i=0}^d; A^*; \lbrace E^*_i \rbrace_{i=0}^d )$ denote a tridiagonal system on $V$.
  \begin{enumerate}
  \item[\rm (i)] 
   Let $0\not=v \in E_0V$. Then the map $\langle A^* \rangle \to V, B \mapsto Bv$ is injective.
   \item[\rm (ii)] Let $0 \not=w \in E_dV$. Then the map $\langle A^* \rangle \to V, B \mapsto Bw$ is injective.
   \end{enumerate}
\end{corollary}
\begin{proof} Apply Lemma \ref{lem:inj}  to $\Phi^*$.
\end{proof}

 \begin{corollary}\label{thm:main1d} Let $\Phi = (A; \lbrace E_i \rbrace_{i=0}^d; A^*; \lbrace E^*_i \rbrace_{i=0}^d )$ denote a tridiagonal system on $V$.
Then the following are equivalent:
\begin{enumerate}
\item[\rm (i)] there exists a nonzero $v \in E_0V$ and  nonzero $S^* \in \langle A^* \rangle$ such that $S^*v \in E_dV$;
\item[\rm (ii)] $\Phi$ is a Leonard system.
\end{enumerate}
\noindent Assume that {\rm (i), (ii)} hold. Then $S^* E_0W=E_dV$.
\end{corollary}
\begin{proof}
Apply Proposition \ref{thm:main1}  to $\Phi^*$.
 \end{proof}

\section{$Q$-polynomial distance-regular graphs}
We now turn our attention to graph theory. For the rest of this paper, we assume that $\mathbb F=\mathbb C$. We will use the following notation.
Let $X$ denote a nonempty finite set.
Let ${\rm Mat}_X(\mathbb C)$ denote the $\mathbb C$-algebra consisting of the matrices that have rows and columns indexed by $X$
and all entries in $\mathbb C$. 
Let $V=\mathbb C^X $ denote the vector space over $\mathbb C$, consisting of the column vectors that have coordinates indexed by $X$ and all entries in $\mathbb C$.
The algebra ${\rm Mat}_X(\mathbb C)$ acts on $V$ by left multiplication. We call $V$ the {\it standard module}.
We endow $V$ with a Hermitean form $\langle \,,\,\rangle$ such that
$\langle u,v\rangle = u^{\rm t} {\overline v}$ for all $u,v \in V$. Here $\rm t$ denotes transpose and $-$ denotes complex conjugation.
We comment on the Hermitean form. For $u,v \in V$ and $B \in {\rm Mat}_X(\mathbb C)$,
\begin{align} 
\label{eq:AABil}
\langle Bu, v \rangle = \langle u, {\overline B}^{\rm t} v\rangle.
\end{align}
For $x \in X$ define a vector  ${\hat x} \in V$ that has $x$-coordinate $1$
and all other coordinates $0$. The vectors $\lbrace \hat x \vert x \in X\rbrace$ form an orthonormal basis for $V$. For a nonempty subset
$C \subseteq X$ the vector $\sum_{x \in C} {\hat x}$ is called the {\it characteristic vector of $C$}.
\medskip

\noindent
Let $\Gamma=(X, \mathcal R)$ denote a finite, undirected graph, without loops or multiple edges, with vertex set $X$ and adjacency relation $\mathcal R$.
For the time being, we do not assume that $\Gamma$ is connected.
For $x \in X$ let $\Gamma(x)$ denote the set of vertices in $X$ that are adjacent to $x$.
We say that $\Gamma$ is {\it regular with valency $k$} whenever the cardinality $\vert \Gamma(x)\vert=k$ for all $x \in X$.
\medskip

\noindent Define a matrix $A \in {\rm Mat}_X(\mathbb C)$ that has $(x,y)$-entry
\begin{align*}
A_{x,y}= \begin{cases} 1, & \mbox{if $x, y$ are adjacent}; \\
                                            0, & \mbox{if $x,y$ are not adjacent}
                         \end{cases}          \qquad \qquad (x,y \in X).
\end{align*}
\noindent We call $A$ the {\it adjacency matrix} of $\Gamma$. The matrix $A$ is symmetric with real entries, so $A$ is diagonalizable over $\mathbb R$.
By the {\it eigenvalues of $\Gamma$} we mean the roots of the minimal polynomial of $A$. These eigenvalues are real and mutually distinct.
Later in the paper, we will use the following basic result.

\begin{lemma} \label{lem:REG} {\rm (See \cite[Section~3]{biggs}.)} Assume that $\Gamma$ is regular with valency $k$. Then the following {\rm (i)--(iii)} hold:
\begin{enumerate}
\item[\rm (i)]
 $k$ is the maximal eigenvalue of $\Gamma$;
 \item[\rm (ii)]  the $k$-eigenspace of $A$ has an orthogonal basis consisting of the characteristic vectors of the  connected components of $\Gamma$;
\item[\rm (iii)]   the $k$-eigenspace of $A$ has dimension equal to the number of connected components of $\Gamma$.
\end{enumerate}
\end{lemma}

\noindent Next, we recall the distance-regular property.
For the rest of this paper, assume that $\Gamma$ is connected. Let $\partial$ denote the path-length distance function for $\Gamma$,
and recall the diameter
 \begin{align*}
 D={\rm max} \lbrace \partial(x,y) \vert x,y \in X\rbrace.
 \end{align*}
 \noindent For $x \in X$ and $0 \leq i \leq D$ define the set $\Gamma_i(x) = \lbrace y \in X \vert \partial(x,y)=i\rbrace$. So $\Gamma_0(x)=\lbrace x \rbrace$, 
 and $\Gamma_1(x) = \Gamma(x)$ if $D\geq 1$.
The graph $\Gamma$ is said to be  {\it distance-regular} whenever for  $0 \leq h,i,j\leq D$ and $x,y \in X$ at distance $\partial(x,y)=h$,
the number
\begin{align*}
p^h_{i,j}  = \vert \Gamma_i(x) \cap \Gamma_j(y)\vert
\end{align*}
is a constant that  is independent of $x,y$. The parameter $p^h_{i,j}$ is called an {\it intersection number} of $\Gamma$.
Background information about distance-regular graphs can be found in \cite{Nbbit, Nbannai, Nbcn, Ndkt, Nint}.
\medskip

\noindent 
 For the rest of this paper,
we assume that $\Gamma$ is distance-regular with $D\geq 1$.
By the triangle inequality, the following hold for $0 \leq h,i,j\leq D$:
\begin{enumerate}
\item[\rm (i)] $p^h_{i,j}= 0$ if one of $h,i,j$ is greater than the sum of the other two;
\item[\rm (ii)] $p^h_{i,j}\not=0$ if one of $h,i,j$ is equal to the sum of the other two.
\end{enumerate}
\noindent Abbreviate
\begin{align*}
c_i = p^i_{1,i-1} \; (1 \leq i \leq D), \qquad
a_i = p^i_{1,i} \; (0 \leq i \leq D), \qquad
b_i = p^i_{1,i+1} \; (0 \leq i \leq D-1).
\end{align*}
By construction, $a_0=0$ and $c_1=1$. The graph $\Gamma$ is regular with valency $k=b_0$. We have
\begin{align*}
c_i + a_i + b_i = k \qquad \qquad (0 \leq i \leq D),
\end{align*}
where $c_0=0$ and $b_D=0$. 
\medskip

\noindent Next, we recall the valencies of $\Gamma$. For $0 \leq i \leq D$, abbreviate $k_i=p^0_{i,i}$ and note that $k_i = \vert \Gamma_i(x) \vert$ for all $x \in X$.
We call $k_i$ the $i$th {\it valency} of $\Gamma$. We have $k_0=1$ and $k_1=k$. By a routine counting argument, $k_i c_i = k_{i-1} b_{i-1}$  $(1 \leq i \leq D)$.
Therefore,
\begin{align} \label{eq:ki}
k_i = \frac{b_0 b_1 \cdots b_{i-1}}{c_1 c_2 \cdots c_i} \qquad \qquad (0 \leq i \leq D).
\end{align}
\noindent Next, we recall the Bose-Mesner algebra of $\Gamma$.
For 
$0 \leq i \leq D$ define $A_i \in {\rm Mat}_X(\mathbb C)$ that has
$(x,y)$-entry
\begin{align*}
(A_i)_{x,y} = \begin{cases}  
1, & {\mbox{\rm if $\partial(x,y)=i$}};\\
0, & {\mbox{\rm if $\partial(x,y) \ne i$}}
\end{cases}
 \qquad \qquad (x,y \in X).
\end{align*}
  \noindent
We call $A_i$ the $i$th {\it distance matrix} of $\Gamma$. We have $A_0=I$ and  $A_1=A$. Let $J\in {\rm Mat}_X(\mathbb C)$ have all entries $1$.
Observe that
 (i)
$J=\sum_{i=0}^D A_i $;
(ii)  $A_i^{\rm t} = A_i  \;(0 \leq i \leq D)$;
(iii)  $\overline{A_i} = A_i  \;(0 \leq i \leq D)$;
(iv) $A_iA_j = \sum_{h=0}^D p^h_{i,j} A_h \;( 0 \leq i,j \leq D) $.
Therefore, the matrices
 $\lbrace A_i\rbrace_{i=0}^D$
form a basis for a commutative subalgebra $M$ of ${\rm Mat}_X(\mathbb C)$, called the 
{\it Bose-Mesner algebra} of $\Gamma$.
The matrix $A$ generates $M$ \cite[Corollary~3.4]{Nint}. 
We mentioned earlier that $A$ is diagonalizable over $\mathbb R$.
 Consequently
$M$ has a second basis 
$\lbrace E_i\rbrace_{i=0}^D$ such that
(i) $E_0 = |X|^{-1}J$;
(ii) $I=\sum_{i=0}^D E_i$;
(iii) $E_i^{\rm t} =E_i  \;(0 \leq i \leq D)$;
(iv) $\overline{E_i} =E_i \;(0 \leq i \leq D)$;
(v) $E_iE_j =\delta_{i,j}E_i  \;(0 \leq i,j \leq D)$.
We call $\lbrace E_i\rbrace_{i=0}^D$  the {\it primitive idempotents}
of $\Gamma$. 
\medskip 

\noindent By construction,
\begin{align}\label{eq:esd}
V = \sum_{i=0}^D E_iV \qquad \qquad {\mbox{\rm (orthogonal direct sum)}}.
\end{align}
The summands in \eqref{eq:esd} are the eigenspaces of $A$. For $0 \leq i \leq D$ let $\theta_i$ denote the eigenvalue of $A$ for $E_iV$. 
We have $AE_i = \theta_i E_i = E_i A$ and $A=\sum_{i=0}^D \theta_i E_i$. By \cite[p.~128]{Nbcn} we have $\theta_0=k$.
\medskip

 \noindent Next, we recall the Krein parameters of $\Gamma$. We have  $A_i \circ A_j = \delta_{i,j} A_i$ $(0 \leq i,j\leq D)$, where $\circ $ denotes entry-wise multiplication.
 Therefore,  $M$  is closed under $\circ$.
Consequently, there exist scalars $q^{h}_{i,j} \in \mathbb C $ $(0 \leq h,i,j\leq D)$ such that
\begin{align*} 
E_i \circ E_j = \vert X \vert^{-1} \sum_{h=0}^D q^h_{i,j} E_h \qquad \qquad (0 \leq i,j\leq D).
\end{align*}  
The scalars $q^{h}_{i,j}$ are called the {\it Krein parameters} of $\Gamma$. The Krein parameters 
are real and nonnegative \cite[p.~69]{Nbannai}.
\medskip

  \noindent Next, we recall the $Q$-polynomial property. The graph $\Gamma$ is said to be {\it $Q$-polynomial} (with respect to the
   ordering $\lbrace E_i \rbrace_{i=0}^D$)  whenever the following hold
for $0 \leq h,i,j\leq D$:
\begin{enumerate}
\item[\rm (i)] $q^h_{i,j}=0$ if one of $h,i,j$ is greater than the sum of the other two;
\item[\rm (ii)]  $q^h_{i,j}\not=0$ if one of $h,i,j$ is equal to the sum of the other two.
\end{enumerate}
For the rest of this paper, we  assume that $\Gamma$ is $Q$-polynomial with respect to the ordering $\lbrace E_i \rbrace_{i=0}^D$.
 \medskip
 
 \noindent
Next, we  recall the dual Bose-Mesner algebras of $\Gamma$.
For the rest of this section, fix
a vertex $x \in X$.
For 
$ 0 \leq i \leq D$ define a diagonal matrix $E_i^*=E_i^*(x)$  in 
 ${\rm Mat}_X(\mathbb C)$
 that has $(y,y)$-entry
\begin{align*}
(E_i^*)_{y,y} = \begin{cases} 1, & \mbox{\rm if $\partial(x,y)=i$};\\
0, & \mbox{\rm if $\partial(x,y) \ne i$}
\end{cases}
 \qquad \qquad (y \in X).
\end{align*}
We call $E_i^*$ the  $i$th {\it dual primitive idempotent of $\Gamma$
 with respect to $x$} \cite[p.~378]{tSub1}. Note that
(i) $I=\sum_{i=0}^D E_i^*$;
(ii) $(E_i^*)^{\rm t} = E_i^*$ $(0 \leq i \leq D)$;
(iii) $\overline{E^*_i} = E_i^*$ $(0 \leq i \leq D)$;
(iv) $E_i^*E_j^* = \delta_{i,j}E_i^* $ $(0 \leq i,j \leq D)$.
Therefore, the matrices
$\lbrace E_i^*\rbrace_{i=0}^D$ form a 
basis for a commutative subalgebra
$M^*=M^*(x)$ of 
${\rm Mat}_X(\mathbb C)$.
We call 
$M^*$ the {\it dual Bose-Mesner algebra of
$\Gamma$ with respect to $x$} \cite[p.~378]{tSub1}.
\medskip

\noindent Next, we recall the dual distance matrices of $\Gamma$.
  For $0 \leq i \leq D$ define a diagonal matrix $A^*_i = A^*_i(x) $  in ${\rm Mat}_X(\mathbb C)$
that has $(y,y)$-entry
\begin{align*}
 (A^*_i)_{y,y} = \vert X \vert (E_i)_{x,y} \qquad \qquad (y \in X).
 \end{align*}
 By \cite[Lemma~5.8]{Nint} the matrices $\lbrace A^*_i \rbrace_{i=0}^D$ form a basis for $M^*$. We have
 (i) $A^*_0=I$;
(ii) $\sum_{i=0}^D A^*_i = \vert X \vert E^*_0$;
(iii)  $(A^*_i)^{\rm t} = A^*_i \;(0 \leq i \leq D)$;
(iv)  $\overline{A^*_i}= A^*_i \;(0 \leq i \leq D)$;
(v) $A^*_i A^*_j = \sum_{h=0}^D q^h_{i,j} A^*_h\; (0 \leq i,j\leq D)$.
We call $\lbrace A^*_i \rbrace_{i=0}^D$ the  {\it dual distance matrices of $\Gamma$ with respect to $x$}. We abbreviate $A^*=A^*_1$ and call this
the {\it dual adjacency matrix of $\Gamma$ with respect to $x$}. By \cite[Corollary~11.6]{Nint}, $A^*$ generates $M^*$.
\medskip

\noindent By construction,
\begin{align*}
E^*_iV = {\rm Span}\lbrace {\hat y} \vert y \in \Gamma_i(x) \rbrace \qquad\qquad (0 \leq i \leq D).
\end{align*}
\noindent Note that
\begin{align} \label{eq:AsED}
V = \sum_{i=0}^D E^*_iV \qquad \qquad {\mbox{\rm (orthogonal direct sum)}}.
\end{align}
\noindent The summands in \eqref{eq:AsED} are the eigenspaces of $A^*$.
 For $0 \leq i \leq D$ let $\theta^*_i$ denote the eigenvalue of $A^*$ for $E^*_iV$. We have $A^*E^*_i = \theta^*_i E^*_i = E^*_i A^*$ and $A^*=\sum_{i=0}^D \theta^*_i E^*_i$. 
The scalars $\lbrace \theta^*_i \rbrace_{i=0}^D$ are real and mutually distinct. We call $\lbrace \theta^*_i \rbrace_{i=0}^D$ the {\it dual eigenvalues of $\Gamma$} (with
respect to the given $Q$-polynomial structure).
\medskip

\noindent
Next, we recall the subconstituent algebras of $\Gamma$.
Let $T=T(x)$ denote the subalgebra of ${\rm Mat}_X(\mathbb C)$ generated by 
$M$ and $M^*$. 
The algebra $T$ is finite-dimensional and noncommutative. We call $T$ the {\it subconstituent algebra}
(or {\it Terwilliger algebra}) {\it of $\Gamma$ 
 with respect to $x$} \cite[Definition 3.3]{tSub1}. Note that $T$ is generated by $A, A^*$.
\medskip

\noindent Next, we recall the triple product relations. By \cite[Lemma~3.2]{tSub1} the following hold for $0 \leq h,i,j\leq D$:
\begin{align*}
&E^*_i A_h E^*_j = 0 \quad \hbox{\rm if and only if}\quad p^h_{i,j} =0; \\
&E_i A^*_h E_j = 0 \quad \hbox{\rm if and only if}\quad q^h_{i,j} =0.
\end{align*}

\noindent Among these relations, we emphasize the cases $h=1$ and $h=D$. 
For $0 \leq i,j\leq D$:
\begin{align*}
E^*_i A E^*_j &= \begin{cases} 0, & \mbox{if $\vert i-j\vert >1$}; \\
                                            \not=0, & \mbox{if $\vert i-j\vert=1$};
                         \end{cases} \\
E_i A^* E_j &= \begin{cases} 0, & \mbox{if $\vert i-j\vert >1$}; \\
                                            \not=0, & \mbox{if $\vert i-j\vert=1$}.
                         \end{cases} 
\end{align*}
For $0 \leq i,j\leq D$:
\begin{align*}
E^*_i A_D E^*_j &= \begin{cases} 0, & \mbox{if $i+j<D$}; \\
                                            \not=0, & \mbox{if $i+j=D$};
                         \end{cases} \\
E_i A_D^* E_j &= \begin{cases} 0, & \mbox{if $i+j<D$}; \\
                                            \not=0, & \mbox{if $i+j=D$}.
                         \end{cases} 
\end{align*}

\noindent We mention a fact for later use.
\begin{lemma} \label{lem:ADinv} {\rm (See \cite[Section~1]{mamart}.)} The matrices $A_D$ and $A^*_D$ are invertible.
\end{lemma}


\section{Modules for the subconstituent algebra $T$}
We continue to discuss the $Q$-polynomial distance-regular graph $\Gamma=(X, \mathcal R)$. Throughout this section, we fix $x \in X$
and write $T=T(x)$.  We discuss how the standard module $V$ is an orthogonal direct sum of irreducible $T$-modules.
We explain how the pair $A, A^*$ acts on each irreducible $T$-module as a tridiagonal pair.
\medskip

\noindent For convenience, we adopt the following convention.  By a {\it $T$-module}, we mean a $T$-submodule of $V$.
\medskip

\noindent In the next few paragraphs, we recall some concepts from linear algebra. 
\medskip

\noindent Let $W$ denote a subspace of $V$. Define the subspace
\begin{align*}
W^\perp = \lbrace v \in V \vert \langle v,w \rangle=0 \rbrace.
\end{align*}
By linear algebra, 
\begin{align*}
V = W + W^\perp \qquad \qquad \hbox{(orthogonal direct sum).}
\end{align*}
We call $W^\perp$ the {\it orthogonal complement of $W$ in $V$}. 
\medskip

\noindent Let $W$ denote a subspace of $V$, and let $U$ denote a subspace of $W$. By linear algebra, 
\begin{align*}
W = U + U^\perp \cap W \qquad \qquad \hbox{(orthogonal direct sum).}
\end{align*}
\noindent We call $U^\perp \cap W$ the {\it orthogonal complement of $U$ in $W$}.
\medskip

\noindent  Let $W$ denote a $T$-module, and let $U$ denote a $T$-submodule of $W$.  Using \eqref{eq:AABil} we routinely find that the
 orthogonal complement of $U$ in $W$ is a $T$-module.
 
\begin{definition} \rm A $T$-module $W$ is said to be {\it irreducible} whenever $W \not=0$ and
$W$ does not contain a $T$-module besides $0$ and $W$.
\end{definition}

\begin{lemma} \label{lem:ModODS} {\rm (See \cite[Section~1, Lemma~3.4]{tSub1}.)} Every $T$-module is an orthogonal direct sum of irreducible $T$-modules.
In particular, the standard $T$-module $V$ is an orthogonal direct sum of irreducible $T$-modules.
\end{lemma}

\noindent We recall the notion of isomorphism for $T$-modules. Let $W$ and $W'$ denote  $T$-modules. By a {\it $T$-module isomorphism from $W$ to $W'$}, we mean
a $\mathbb C$-linear bijection $\sigma : W \to W'$ such that $ \sigma B =B \sigma $ on $W$ for all $B\in T$. The $T$-modules $W$ and $W'$ are called {\it isomorphic} whenever
there exists a $T$-module isomorphism from $W$ to $W'$.

\begin{lemma} \label{lem:orth} {\rm (See \cite[Lemma~3.3]{curtin}.)} Nonisomorphic irreducible $T$-modules are orthogonal.
\end{lemma}
\medskip

\noindent Let $W$ denote an irreducible $T$-module. We mention some parameters attached to $W$.
By linear algebra, $W$ is the orthogonal direct sum of the nonzero subspaces among $\lbrace E^*_iW\rbrace_{i=0}^D$.
By the {\it endpoint} of $W$ we mean
\begin{align*}
{\rm min}\lbrace i \vert 0 \leq i \leq D, \; E^*_iW \not=0 \rbrace.
\end{align*}
\noindent By the {\it diameter} of $W$, we mean
\begin{align*}
\bigl \vert \lbrace i \vert 0 \leq i \leq D, \; E^*_iW \not=0 \rbrace  \bigr\vert -1.
\end{align*}
\noindent  Note that $W$ is the orthogonal direct sum of the nonzero subspaces among $\lbrace E_iW\rbrace_{i=0}^D$.
By the {\it dual endpoint} of $W$ we mean
\begin{align*}
{\rm min}\lbrace i \vert 0 \leq i \leq D, \; E_iW \not=0 \rbrace.
\end{align*}
\noindent By the {\it dual diameter} of $W$, we mean
\begin{align*}
\bigl \vert \lbrace i \vert 0 \leq i \leq D, \; E_iW \not=0 \rbrace  \bigr\vert -1.
\end{align*}

\begin{lemma} \label{lem:rtd} {\rm (See \cite[Lemmas~3.9, 3.12]{tSub1}.)} Let $W$ denote an irreducible $T$-module, with endpoint $r$, dual endpoint $t$, diameter $\delta$, and dual diameter $d$.
Then:
\begin{enumerate}
\item[\rm (i)] $t,d$ are nonnegative and $t+d\leq D$;
\item[\rm (ii)] $E_iW \not=0$ if and only if $ t \leq i \leq t+d$ $(0 \leq i \leq D)$;
\item[\rm (iii)] $r, \delta$ are nonnegative and $r+\delta\leq D$;
\item[\rm (iv)] $E^*_iW \not=0$ if and only if $ r \leq i \leq r+\delta$ $(0 \leq i \leq D)$.
\end{enumerate}
\end{lemma}

\begin{lemma} \label{lem:AAW} {\rm (See \cite[Lemmas~3.9, 3.12]{tSub1}.)} Let $W$ denote an irreducible $T$-module, with endpoint $r$, dual endpoint $t$, diameter $\delta$, and dual diameter $d$.
\begin{enumerate}
\item[\rm (i)]  Define $W_i = E_{t+i} W$ for $0 \leq i \leq d$. Then $\lbrace W_i \rbrace_{i=0}^d$ is an ordering of the eigenspaces of $A$ on $W$. Moreover
\begin{align*}
        A^* W_i \subseteq W_{i-1} + W_i + W_{i+1} \qquad \quad (0 \leq i \leq d),
\end{align*}
where $W_{-1}=0$ and $W_{d+1}=0$.
\item[\rm (ii)] Define $W^*_i = E^*_{r+i} W$ for $0 \leq i \leq \delta$. Then $\lbrace W^*_i \rbrace_{i=0}^\delta$ is an ordering of the eigenspaces of $A^*$ on $W$. Moreover
\begin{align*}
        A W^*_i \subseteq W^*_{i-1} + W^*_i + W^*_{i+1} \qquad \quad (0 \leq i \leq \delta),
\end{align*}
where $W^*_{-1}=0$ and $W^*_{\delta+1}=0$.
\end{enumerate}
\end{lemma}

\begin{lemma} \label{lem:WTD} {\rm (See \cite[Example~1.4]{NsomeAlg}, \cite[Corollary~3.3]{pasc}.)} Let $W$ denote an irreducible $T$-module, with endpoint $r$, dual endpoint $t$, diameter $\delta$, and dual diameter $d$. Then:
\begin{enumerate}
\item[\rm (i)] the pair $A, A^*$ acts on $W$ as a tridiagonal pair;
\item[\rm (ii)] $d=\delta$;
\item[\rm (iii)]  the sequence $(A; \lbrace E_{t+i} \rbrace_{i=0}^d; A^*; \lbrace E^*_{r+i} \rbrace_{i=0}^d)$ acts on $W$ as a tridiagonal system.
\end{enumerate}
\end{lemma}

\noindent  Let $W$ denote an irreducible $T$-module. We have seen that $A, A^*$ act on $W$ as a tridiagonal pair. In the next few results, we apply
the theory of tridiagonal pairs to $W$.

\begin{lemma} \label{lem:shape} {\rm (See \cite[Corollary~5.7]{NsomeAlg}.)} Let $W$ denote an irreducible $T$-module, with endpoint $r$, dual endpoint $t$, and diameter $d$.
Then for $0 \leq i \leq d$ the following subspaces have the same dimension:
\begin{align}
 E_{t+i}W, \qquad E_{t+d-i}W, \qquad 
E^*_{r+i}W, \qquad E^*_{r+d-i}W. \label{eq:4sub}
\end{align}
\end{lemma}

\begin{definition}\label{def:Wshape} \rm We refer to the irreducible $T$-module $W$ from Lemma \ref{lem:shape}. For $0 \leq i \leq d$ let $\rho_i$ denote
the common dimension of the four subspaces \eqref{eq:4sub}. By construction $\rho_i \not=0$ and $\rho_i=\rho_{d-i}$. We call the sequence $\lbrace \rho_i \rbrace_{i=0}^d$ the {\it shape} of $W$.
\end{definition}

\begin{lemma} \label{lem:shapeFacts} {\rm (See \cite[Corollary~6.6]{NsomeAlg}.)} Let $W$ denote an irreducible $T$-module, with diameter $d$ and shape  $\lbrace \rho_i \rbrace_{i=0}^d$. Then
 $\rho_{i-1} \leq  \rho_i$ for $1 \leq i \leq d/2$.
\end{lemma}

\begin{lemma} \label{lem:WSP} {\rm (See \cite[Theorem 4.6, Corollary 5.7]{NsomeAlg}.)} Let $W$ denote an irreducible $T$-module, with endpoint $r$, dual endpoint $t$, and diameter $d$.
Then the following sum is direct:
\begin{align*}
W=\sum_{i=0}^d \Bigl( (E^*_rW + \cdots + E^*_{r+i}W ) \cap (E_tW+\cdots + E_{t+d-i}W) \Bigr).
\end{align*}
Moreover, for $0 \leq i \leq d$ the $i$th summand has dimension $\rho_i$, where $\lbrace \rho_i \rbrace_{i=0}^d$ is the shape of $W$.
\end{lemma}

\begin{definition}\label{def:TDthin} \rm (See \cite[Definition~3.5]{tSub1}.) Let $W$ denote an irreducible $T$-module, with diameter $d$ and shape $\lbrace \rho_i \rbrace_{i=0}^d$. Then $W$ is called {\it thin} whenever
$\rho_i=1$ for $0 \leq i \leq d$. 
\end{definition}

\begin{lemma} \label{lem:THandLS} Let $W$ denote an irreducible $T$-module. Then $W$ is thin if and only if the tridiagonal system in Lemma \ref{lem:WTD}(iii)  is a Leonard system.
\end{lemma}
\begin{proof} By Definition \ref{def:TDthin} and  \cite[Definition~1.4]{LS99}.
\end{proof}

\begin{example}\label{ex:primary} \rm (See \cite[Lemma~3.6]{tSub1}, \cite[Section~7]{Nint}.) There exists a unique irreducible $T$-module that has diameter $D$; this $T$-module is called {\it primary}. The  primary $T$-module is thin.
An irreducible $T$-module is primary iff it has endpoint 0 iff it has dual endpoint 0.
\end{example}

\section{The nucleus, I}

\noindent We continue to discuss the $Q$-polynomial distance-regular graph $\Gamma=(X,\mathcal R)$. Throughout this section, we fix $x \in X$ and write $T=T(x)$.
We introduce a $T$-module called the nucleus, and describe its basic properties.

\begin{lemma} \label{lem:prelim} Let $W$ denote an irreducible $T$-module, with endpoint $r$ and diameter $d$. Then
\begin{align*}
0 \not=A_D E^*_rW \subseteq \sum_{i=D-r}^{r+d} E^*_iW.
\end{align*}
\end{lemma}
\begin{proof} The subspace  $A_D E^*_rW$ is nonzero, because $A_D$ is invertible and $E^*_r W$ is nonzero. 
We have
\begin{align} \label{eq:c1}
A_D E^*_r W \subseteq TW \subseteq W = \sum_{i=r}^{r+d} E^*_iW.
\end{align}
Using $I = \sum_{i=0}^D E^*_i$ and the discussion above Lemma \ref{lem:ADinv},
\begin{align} \label{eq:c2}
A_D E^*_rW = I A_D E^*_r W = \sum_{i=0}^D E^*_i A_D E^*_r W = \sum_{i=D-r}^D E^*_i A_D E^*_rW\subseteq \sum_{i=D-r}^D E^*_iW.
\end{align}
Combining \eqref{eq:c1}, \eqref{eq:c2} we obtain $ A_D E^*_rW \subseteq \sum_{i=D-r}^{r+d} E^*_iW$.
\end{proof} 

\noindent In the following proposition, item (i) was obtained earlier in \cite[Lemma~5.1]{caughman}. We will give a new proof that illuminates the case of equality.

\begin{proposition} \label{lem:Wineq}  Let $W$ denote an irreducible $T$-module, with endpoint $r$  and diameter $d$.
Then the following hold.
\begin{enumerate}
\item[\rm (i)] {\rm (See \cite[Lemma~5.1]{caughman}.)} $2r-D+d\geq 0$.
\item[\rm (ii)] Assume that equality holds in {\rm (i)}. Then $W$ is thin and $A_DE^*_rW=E^*_{D-r}W$.
\end{enumerate}
\end{proposition}
\begin{proof} (i) By Lemma \ref{lem:prelim} we have $r+d \geq D-r$. \\
\noindent (ii)
We apply Proposition \ref{thm:main1} to the tridiagonal system in Lemma \ref{lem:WTD}(iii). Note that $A_D \in M= \langle A \rangle$.
 By Lemma \ref{lem:prelim} we have $0 \not=A_D E^*_rW \subseteq E^*_{D-r}W$. 
Pick $0 \not=v \in E^*_rW$. 
Define $S \in {\rm End}(W)$ to be the restriction of $A_D$ to $W$. 
By construction $S \not=0$ and $Sv \in E^*_{D-r}W$. By this and Proposition \ref{thm:main1}, we find that the tridiagonal system in Lemma \ref{lem:WTD}(iii)
 is a Leonard system. Consequently $W$ is thin in view of Lemma \ref{lem:THandLS}. Also by Proposition \ref{thm:main1},
 $A_DE^*_rW=E^*_{D-r}W$.
\end{proof}

\begin{lemma} \label{lem:prelim2} Let $W$ denote an irreducible $T$-module, with dual endpoint $t$ and diameter $d$. Then
\begin{align*}
0 \not=A^*_D E_tW \subseteq \sum_{i=D-t}^{t+d} E_iW.
\end{align*}
\end{lemma}
\begin{proof} Similar to the proof of Lemma \ref{lem:prelim}.
\end{proof}

\begin{proposition} \label{lem:Wineq2}  Let $W$ denote an irreducible $T$-module, with dual endpoint $t$ and diameter $d$.
Then the following hold.
\begin{enumerate}
\item[\rm (i)] {\rm (See \cite[Lemma~7.1]{caughman}.)} $2t-D+d\geq 0$.
\item[\rm (ii)] Assume that equality holds in {\rm (i)}. Then $W$ is thin and $A^*_D E_tW=E_{D-t}W$.
\end{enumerate}
\end{proposition}
\begin{proof} Similar to the proof of Proposition \ref{lem:Wineq}, using Corollary \ref{thm:main1d} and Lemma \ref{lem:prelim2} instead of Proposition \ref{thm:main1}  and Lemma \ref{lem:prelim}.
\end{proof}

\begin{proposition} \label{prop:mix} Let $W$ denote an irreducible $T$-module, with endpoint $r$, dual endpoint $t$, and diameter $d$.
Then 
\begin{align}
r+t-D+d \geq 0. \label{eq:disp}
\end{align}
Moreover, equality holds in \eqref{eq:disp} if and only if both $t=r$ and $d=D-2r$. In this case, $W$ is thin and $A_D E^*_rW=E^*_{D-r}W$ and $A^*_D E_rW=E_{D-r}W$. 
\end{proposition}
\begin{proof} Note that
\begin{align*}
r+t-D+d = \frac{2r-D+d}{2} + \frac{2t-D+d}{2}.
\end{align*}
The result follows from this and Propositions \ref{lem:Wineq}, \ref{lem:Wineq2}.
\end{proof}

\begin{definition}\label{def:DISP} \rm (See \cite[Definition~4.1]{ds}.) Let $W$ denote an irreducible $T$-module. By the {\it displacement of $W$} we mean the integer
\begin{align*}
         r+t-D+d,
\end{align*}
where $r$ (resp. $t$) (resp. $d$)  denotes the endpoint (resp. dual endpoint) (resp. diameter) of $W$.
\end{definition}

\begin{example}\rm The primary $T$-module has displacement 0.
\end{example}

\begin{definition}\label{def:NUC} \rm By the {\it nucleus} of $\Gamma$ with respect to $x$, we mean the span of
the irreducible $T$-modules that have displacement 0.
\end{definition}


\begin{lemma} \label{lem:Interp} Consider the nucleus of $\Gamma$ with respect to $x$.
\begin{enumerate}
\item[\rm (i)] The nucleus is a $T$-module.
\item[\rm (ii)]  The orthogonal complement of the nucleus in $V$ is spanned by the irreducible $T$-modules that have displacement at least one.
\end{enumerate}
\end{lemma}
\begin{proof} (i) The span of two $T$-modules is a $T$-module.\\
\noindent (ii) An irreducible $T$-module with zero displacement is not isomorphic to
an irreducible $T$-module with nonzero displacement.
The result follows in view of Lemmas \ref{lem:ModODS}, \ref{lem:orth} and Definition \ref{def:NUC}. 
\end{proof}

\noindent We emphasize a few points about the nucleus.
\begin{proposition} \label{lem:Wdata} Let $W$ denote an irreducible $T$-submodule of the nucleus, with endpoint $r$, dual endpoint $t$, and diameter $d$.
Then $0 \leq r \leq D/2$ and $t=r$ and $d=D-2r$. Moreover, $W$ is thin and $A_D E^*_rW=E^*_{D-r}W$ and $A^*_D E_rW=E_{D-r}W$.
\end{proposition}
\begin{proof} By Proposition \ref{prop:mix} and Definitions \ref{def:DISP}, \ref{def:NUC} along with Lemma \ref{lem:Interp}. 
\end{proof}

\begin{proposition}\label{lem:nucISO} Let $W, W'$ denote irreducible $T$-submodules of the nucleus. Then the following are equivalent:
\begin{enumerate}
\item[\rm (i)] the endpoints of $W, W'$ are the same;
\item[\rm (ii)] the $T$-modules $W, W'$ are isomorphic.
\end{enumerate}
\end{proposition}
\begin{proof}
\noindent ${\rm (i)}\Rightarrow {\rm (ii)}$ 
Denote the common endpoint by $r$. By Proposition \ref{lem:Wdata} we have $0 \leq r\leq D/2$. Also by Proposition \ref{lem:Wdata},  $d=D-2r$ is the common diameter of $W, W'$.
The restriction of $A$ to $W$ or $W'$ has eigenvalues $\theta_r, \theta_{r+1}, \ldots, \theta_{r+d}$.  Therefore,
the following holds on $W$ and $W'$:
\begin{align}
0 = (A-\theta_r I)(A-\theta_{r+1}I) \cdots (A-\theta_{r+d}I). \label{eq:mpA}
\end{align}
 Define the polynomials
 \begin{align}
 f_i(\lambda) = (\lambda-\theta_{r+d})(\lambda-\theta_{r+d-1}) \cdots (\lambda-\theta_{r+d-i+1}) \qquad  (0 \leq i \leq d). \label{eq:mpA2}
 \end{align}
 The polynomial $f_i$ is monic with degree $i$ $(0 \leq i \leq d)$.
 There exists a polynomial $f(\lambda)$ such that $A_D=f(A)$.
 By linear algebra, there exists a unique sequence of complex scalars $\lbrace \alpha_i \rbrace_{i=0}^d$ such that
 \begin{align*}
 f(\theta_{r+j} ) = \sum_{i=0}^d \alpha_i f_i(\theta_{r+j}) \qquad \quad (0 \leq j \leq d).
 \end{align*}
  By construction, the following holds on $W$ and $W'$:
 \begin{align*}
         A_D = \sum_{i=0}^d \alpha_i f_i(A).
 \end{align*}        
 We apply Proposition \ref{thm:main1} and its proof, using $S=A_D$ along with the Leonard systems in Lemma \ref{lem:WTD}(iii)
 associated with $W, W'$. We begin with $W$.
Pick $0 \not=v \in E^*_rW$ and note that  
\begin{align}
\label{eq:short}
A^* v = \theta^*_r v.
\end{align}
 The vectors $\lbrace f_i(A)v \rbrace_{i=0}^d$ form a basis for $W$. 
By \eqref{eq:mpA} and \eqref{eq:mpA2},
\begin{align} \label{eq:short2}
&(A -\theta_{r+d-i}I) f_i(A)v= f_{i+1}(A)v \quad (0 \leq i \leq d-1),  \qquad  (A-\theta_r I)f_d(A)v=0. 
\end{align}
By the first claim in the proof of Proposition \ref{thm:main1},
\begin{align*}
\alpha_i (A^*- \theta^*_{r+i} I) f_i(A)v = \alpha_{i-1} (\theta^*_{r+d} - \theta^*_{r+i-1}) f_{i-1}(A)v \qquad \quad (1 \leq i \leq d).
\end{align*}
By the second claim in the proof of Proposition \ref{thm:main1}, $\alpha_i \not=0$ $(0 \leq i \leq d)$. 
Define
\begin{align*}
     \phi_i = (\theta^*_{r+d} - \theta^*_{r+i-1}) \alpha_{i-1}/\alpha_i  \qquad \quad (1 \leq i \leq d).
     \end{align*}
We have
\begin{align}\label{eq:short3}
(A^* -\theta^*_{r+i}I)  f_i(A)v= \phi_i f_{i-1}(A)v \qquad (1 \leq i \leq d).
\end{align}
We are done discussing $W$; next consider $W'$. Pick $0 \not=v' \in E^*_rW'$. We do for $W',v'$ what we did for $W,v$. The vectors
$\lbrace f_i(A)v' \rbrace_{i=0}^d$ form a basis for $W'$. The equations \eqref{eq:short}--\eqref{eq:short3} hold with
$v$ replaced by $v'$. We can now easily construct a $T$-module isomorphism from $W$ to $W'$.
There exists a $\mathbb C$-linear bijection $\sigma : W \to W'$ that sends $f_i(A)v \mapsto f_i(A)v'$ for $0 \leq i \leq d$.
Comparing \eqref{eq:short}--\eqref{eq:short3} with the corresponding equations for $v'$, we obtain
$(A \sigma - \sigma A)W=0$ and $(A^* \sigma - \sigma A^*)W=0$. It follows that $\sigma$ is a $T$-module isomorphism.
We have shown that the $T$-modules $W,W'$ are isomorphic.
\\
\noindent ${\rm (ii)}\Rightarrow {\rm (i)}$ Clear.
\end{proof}

\section{The nucleus, II}
We continue to discuss the $Q$-polynomial distance-regular graph $\Gamma=(X,\mathcal R)$.
Throughout this section, we fix $x \in X$ and write $T=T(x)$. 
In Definition \ref{def:NUC}, we used the concept of displacement to define a $T$-module called the nucleus.
In the present section, we describe the nucleus from another point of view.

\begin{lemma} \label{lem:zeroInt} {\rm (See \cite[Definition~5.1, Theorem~6.2(iii)]{ds}.)} For $0 \leq i,j\leq D$ such that $i+j<D$, 
\begin{align}
(E^*_0V+E^*_1V+\cdots + E^*_iV) \cap (E_0V+E_1V+\cdots + E_jV)=0.
\end{align}
\end{lemma}

\begin{definition}\label{def:Ni} {\rm (See \cite[Definition~5.1]{ds}.)} For $0 \leq i \leq D$ define a subspace $\mathcal N_i = \mathcal N_i(x)$ by
\begin{align*}
\mathcal N_i = (E^*_0V+E^*_1V+\cdots + E^*_iV) \cap (E_0V+E_1V+\cdots + E_{D-i}V).
\end{align*}
\end{definition}

\begin{lemma} \label{lem:UDir} The sum $\sum_{i=0}^D \mathcal N_i$ is direct.
\end{lemma}
\begin{proof} By \cite[Corollary~5.8]{ds} and Lemma \ref{lem:zeroInt}.
\end{proof}

\begin{definition} \label{def:N} \rm Define a subspace  $\mathcal N = \mathcal N(x)$ by
\begin{align*}
\mathcal N = \sum_{i=0}^D \mathcal N_i.
\end{align*}
\end{definition}

\begin{theorem} \label{thm:ND} The following are the same:
\begin{enumerate}
\item [\rm (i)] the subspace $\mathcal N=\mathcal N(x)$ from Definition \ref{def:N};
\item[\rm (ii)] the nucleus of $\Gamma$ with respect to $x$.
\end{enumerate}
\end{theorem}
\begin{proof} We invoke some results from \cite{ds}. In \cite[Definition~4.3]{ds} the notation $V_\eta$ refers to the subspace of $V$ spanned by the irreducible $T$-modules
with displacement $\eta$. So $V_0$ is the nucleus of $\Gamma$ with respect to $x$. 
In \cite[Definition~5.1]{ds} we defined some subspaces $V_{i,j}$ $(0 \leq i,j\leq D)$  and in 
\cite[Definition~5.5]{ds} we defined some subspaces $\tilde V_{i,j}$ $(0 \leq i,j\leq D)$.
By  \cite[Theorem~6.2(i)]{ds} we have $V_0 = \sum_{i=0}^D \tilde V_{i,D-i}$. 
By \cite[Theorem~5.7]{ds} and \cite[Theorem~6.2(ii)]{ds} we have
$\tilde V_{i,D-i} = V_{i,D-i}$ $(0 \leq i \leq D)$. 
Comparing \cite[Definition~5.1]{ds} and Definition \ref{def:Ni}, we obtain
 $\mathcal N_i = V_{i,D-i}$ $(0 \leq i \leq D)$. 
By these comments and Definition \ref{def:N}, the nucleus of $\Gamma$ with respect to $x$ is equal to
\begin{align*}
V_0 = \sum_{i=0}^D \tilde V_{i,D-i}=  \sum_{i=0}^D V_{i,D-i} = \sum_{i=0}^D \mathcal N_i = \mathcal N.
\end{align*}
\end{proof}

\noindent We have some comments about the subspaces $\lbrace \mathcal N_i\rbrace_{i=0}^D$.
\begin{lemma} \label{lem:AAact} {\rm (See \cite[Theorem~7.1]{ds}.)} We have
\begin{align*}
&(A-\theta_{D-i}I ) \mathcal N_i \subseteq \mathcal N_{i+1} \qquad (0 \leq i \leq D-1), \qquad (A-\theta_0 I) \mathcal N_D=0, \\
&(A^*- \theta^*_i I) \mathcal N_i \subseteq \mathcal N_{i-1} \qquad (1 \leq i \leq D), \qquad (A^*-\theta^*_0 I)\mathcal N_0=0.
\end{align*}
\end{lemma}

\begin{lemma}\label{lem:NiCalc} Let $0 \leq i \leq D$. Then  $\mathcal N_i$ is spanned by the subspaces
\begin{align*}
(E^*_r W + \cdots + E^*_i W) \cap (E_r W+ \cdots + E_{D-i}W),
\end{align*}
where $W$ is an irreducible $T$-module with displacement $0$ and endpoint $r \leq {\rm min}\lbrace i, D-i\rbrace$.
\end{lemma}
\begin{proof} By \cite[Lemma~4.7(i)]{kim2} and \cite[Lemma~6.1]{ds}.
\end{proof}

\begin{lemma} \label{lem:TwoSum} The following sums are orthogonal and direct:
\begin{align*}
\mathcal N = \sum_{i=0}^D E_i \mathcal N,  \qquad \qquad 
\mathcal N = \sum_{i=0}^D E^*_i \mathcal N.
\end{align*}
\end{lemma}
\begin{proof} The subspace $\mathcal N$ is a $T$-module by Lemma \ref{lem:Interp} and Theorem \ref{thm:ND}.
\end{proof}

\noindent To motivate the next result, we have some comments. 
Let $W$ denote an irreducible $T$-submodule of $\mathcal N$. Define the support set
\begin{align*}
{\rm Supp}(W)= \lbrace r, r+1, \ldots, D-r\rbrace \qquad \qquad \hbox{\rm ($r = $ endpoint of $W$).}
\end{align*}
Using Proposition \ref{lem:Wdata} we find that for $0 \leq i \leq D$,
\begin{align*}
{\rm dim}\,E_i W ={\rm dim}\,E_{D-i}W ={\rm dim}\,E^*_i W ={\rm dim}\,E^*_{D-i}W
= \begin{cases} 1, & \mbox{if  $ i \in {\rm Supp}(W)$}; \\
                                             0, & \mbox{if  $ i \not\in {\rm Supp}(W)$}.
                         \end{cases}
\end{align*}
By Lemma \ref{lem:ModODS}, the $T$-module  $\mathcal N$ is an orthogonal direct sum of irreducible $T$-modules; write
\begin{align}
\mathcal N = \sum_W W \qquad  \hbox{\rm  (orthogonal direct sum).}
  \label{eq:NODS}  
  \end{align}
  
  \noindent The following definition is motivated by Proposition \ref{lem:nucISO}.
  
  \begin{definition} \label{def:di} \rm
   For $0 \leq r \leq D/2$,   let ${\rm mult}_r$ denote the number of summands $W$ in \eqref{eq:NODS}  that have endpoint $r$.
   Note that ${\rm mult}_r$ is a nonnegative integer.
\end{definition}
\begin{remark}\rm We have ${\rm mult}_0=1$ by Example \ref{ex:primary} and Definition \ref{def:di}.
\end{remark}

\begin{proposition} \label{prop:Six} For $0 \leq i \leq D/2$, the following subspaces have dimension $\sum_{r=0}^i {\rm mult}_r$:  
\begin{align}
E_i \mathcal N, \qquad
E_{D-i} \mathcal N, \qquad
E^*_i \mathcal N, \qquad
E^*_{D-i} \mathcal N, \qquad 
\mathcal N_i, \qquad \mathcal N_{D-i}.
 \label{eq:5dim}
\end{align}
\end{proposition}
\begin{proof}   To start, we show that the first four subspaces in  \eqref{eq:5dim} have dimension $\sum_{r=0}^i {\rm mult}_r$.
Pick $E \in \lbrace E_i, E_{D-i}, E^*_i, E^*_{D-i}\rbrace$ and
apply $E$ to each side of   \eqref{eq:NODS}. This yields
\begin{align}
E \mathcal N = \sum_W E W \qquad  \hbox{\rm  (orthogonal direct sum).}
  \label{eq:NODS2}  
  \end{align}
Observe that 
\begin{align*}
{\rm dim}\,E \mathcal N = \sum_W {\rm dim}\,EW =  \sum_{\stackrel{W}{ i \in {\rm Supp}(W) }} 1 = \sum_{r=0}^i {\rm mult}_r.
\end{align*}
We have shown that the first four subspaces in  \eqref{eq:5dim} have dimension $\sum_{r=0}^i {\rm mult}_r$.
Similarly, using Lemmas  \ref{lem:WSP}, \ref{lem:NiCalc} we find that $\mathcal N_i$ and $\mathcal N_{D-i}$ have dimension $\sum_{r=0}^i {\rm mult}_r$.
\end{proof}

\section{The nucleus for some elementary examples}
We continue to discuss the $Q$-polynomial distance-regular graph $\Gamma=(X,\mathcal R)$.
Throughout this section, we fix $x \in X$ and write $T=T(x)$. We compute the nucleus of $\Gamma$ with respect to $x$, under the assumption that $\Gamma$
belongs to some elementary families of examples. In the next section, we will consider a more substantial family of examples.

\begin{example} \label{ex:cube} \rm Assume that $\Gamma$ is a $D$-cube \cite{go}. It is shown in \cite[Theorems~6.3,~8.1]{go} that each irreducible $T$-module has
displacement 0. Therefore, the nucleus of $\Gamma$ with respect to $x$ is equal to the standard module $V$.
\end{example}

\noindent The $D$-cube is a bipartite antipodal $2$-cover.

\begin{example}\label{ex:2Hom} \rm Assume that  $\Gamma$ is a bipartite antipodal 2-cover (this property is often  called $2$-homogeneous, see \cite[Theorem 42]{curtin2H}).
It is shown in \cite[Theorem~4.1]{curtin2HT} that each irreducible $T$-module has
displacement 0. Therefore, the nucleus of $\Gamma$ with respect to $x$ is equal to the standard module $V$.
\end{example}

\begin{example}\rm Assume that $\Gamma$ is the Odd graph $O_{D+1}$ \cite[p.~374]{Nbbit}, \cite[Section~6]{tSub3}. It is shown in \cite[Example~6.1(2)]{tSub3}
that for each irreducible $T$-module $W$ the endpoint $r$ and diameter $d$ satisfy $r+d=D$. Consequently, $W$ has displacement 0
if and only if $W$ is primary. Therefore, the nucleus of $\Gamma$ with respect to $x$ is equal to the primary $T$-module.
\end{example}

\begin{example} \rm Assume that $\Gamma$ is a Hamming graph $H(D,N)$ with $N\geq 3$, see \cite{mamart}. 
By construction, the vertex set $X$ of $\Gamma$ has cardinality $N^D$.
Write  $\mathcal N_i = \mathcal N_i(x)$ $(0 \leq i \leq D)$.
By \cite[Lemma~19]{mamart} the subspace $\mathcal N_i$ has dimension $\binom{D}{i}$ for $0 \leq i \leq D$.
By this along with Lemma \ref{lem:UDir}, Definition \ref{def:N} and Theorem \ref{thm:ND}, the nucleus of $\Gamma$ with respect to $x$ has dimension $2^D$.
\end{example}

\section{Dual polar graphs}

In this section, we review a family of $Q$-polynomial distance-regular graphs called the dual polar graphs. 
Our main goal for the rest of the paper, 
is to describe the nucleus of a dual polar graph with respect to any  given vertex.
\medskip
\begin{example} \label{ex:dp} \rm  (See \cite[p.~352]{Nbbit}, \cite[p.~303]{Nbannai}, \cite[Example~6.1(6)]{tSub3}.)
Let $\bf U$ denote a finite vector space with one of the following nondegenerate forms:
\begin{align*}
\begin{tabular}[t]{ccccc}
{\rm name }& ${\rm dim}(\bf U)$ & {\rm field} & form &  $e$
 \\
 \hline
$ B_D(p^n) $ & $2D+1$ & $GF(p^n)$ & quadratic & $0$ \\ 
$ C_D(p^n) $ & $2D$ & $GF(p^n)$ & symplectic & $0$ \\ 
$ D_D(p^n) $ & $2D$ & $GF(p^n)$ & quadratic & $-1$ \\ 
&&& {\rm (Witt index $D$)} & \\
$ {}^2 D_{D+1}(p^n) $ & $2D+2$ & $GF(p^n)$ & quadratic & $1$ \\ 
&&& {\rm (Witt index $D$)} & \\
$ {}^2A_{2D}(p^n) $ & $2D+1$ & $GF(p^{2n})$ & Hermitean & $1/2$ \\ 
$ {}^2A_{2D-1}(p^n) $ & $2D$ & $GF(p^{2n})$ & Hermitean & $-1/2$
    \end{tabular}
\end{align*}
A subspace of $\bf U$ is called {\it isotropic} whenever the form vanishes completely
on that subspace. In each of the above cases, the dimension of any maximal
isotropic subspace is $D$.
The corresponding dual polar graph $\Gamma$ has vertex set $X$ consisting of the maximal isotropic subspaces of $\bf U$. Vertices $y,z \in X$ are adjacent whenever $y \cap z$ has dimension $D-1$. More generally, $\partial(y,z) = D - {\rm dim}\,y \cap z$.
The graph $\Gamma$ is distance-regular with diameter $D$ and intersection numbers 
\begin{align*}
c_i = \frac{q^i-1}{q-1}, \qquad \quad
a_i = (q^{e+1}-1) \frac{q^i-1}{q-1}, \qquad \quad
b_i = q^{e+1}\frac{q^D-q^i}{q-1} 
\end{align*}
for $0 \leq i \leq D$, where $q = p^n, p^n, p^n, p^n, p^{2n}, p^{2n}$. The graph $\Gamma$ is a regular near $2D$-gon in the sense of \cite[Section~6.4]{Nbcn}. 
\end{example}

\begin{remark} \label{rem:U} \rm We refer to Example \ref{ex:dp}. There is a well-known poset consisting of the isotropic subspaces of $\bf U$ and the inclusion relation; 
see for example \cite{artin}, \cite[Section~9.4]{Nbcn}, \cite{Ncameron}, \cite{stanton}. The poset is uniform \cite[Section~3]{uniform}. The poset is a regular quantum matroid \cite[Example~40.1(5)]{qm}.
\end{remark}

\noindent From now on, we assume that $\Gamma$ is a dual polar graph as in Example \ref{ex:dp}. We keep our assumption that $D\geq 1$.
Note that $\Gamma$
is bipartite iff $e=-1$ iff $a_1 =0$. Sometimes we will assume that $\Gamma$ is nonbipartite.

\begin{lemma} \label{ex:dpg2}  {\rm (See \cite[p.~352]{Nbbit}, \cite[p.~303]{Nbannai}, \cite[Example~6.1(6)]{tSub3}.)}
The graph  $\Gamma $ has a $Q$-polynomial structure such that
\begin{align*}
\theta_i &= q^{e+1} \frac{q^D-1}{q-1}-\frac{(q^i-1)(q^{D+e+1-i}+1)}{q-1} \qquad \quad (0 \leq i \leq D), \\
\theta^*_i &= \frac{q^{D+e}+q}{q^e+1} \, \frac{q^{-i}(q^{D+e}+1)-q^e-1}{q-1} \qquad (0 \leq i \leq D).
\end{align*}
\end{lemma}

\begin{remark}\rm (See \cite[Example~20.7]{LSnotes}.)
The $Q$-polynomial structure in Lemma \ref{ex:dpg2}   has dual $q$-Krawtchouk type, with parameters
\begin{align*}
&s = -q^{-D-e-2},  \qquad \quad h = \frac{q^{D+e+1}}{q-1}, \qquad \quad h^*= \frac{(q^{D+e}+1)(q^{D+e}+q)}{(q-1)(q^e+1)}.
\end{align*}
\end{remark}

\begin{remark}\rm (See \cite[p.~194]{Nbcn}.) The $Q$-polynomial structure in Lemma \ref{ex:dpg2}  has classical parameters $(D,b,\alpha, \beta)$ with
\begin{align*}
b=q, \qquad \quad \alpha=0, \qquad \quad \beta = q^{e+1}.      
\end{align*}
\end{remark}

\noindent Note that $q^{e+1}=a_1+1$. It is often convenient to write things in terms of $a_1$ instead of $e$. 

\begin{lemma} \label{lem:inter} The intersection numbers of $\Gamma$ are given by
\begin{align*}
c_i = \frac{q^i-1}{q-1}, \qquad \quad
a_i = a_1 \frac{q^i-1}{q-1}, \qquad \quad
b_i = (a_1+1)\frac{q^D-q^i}{q-1}
\end{align*}
\noindent for $0 \leq i \leq D$. 
\end{lemma} 

\begin{lemma} \label{lem:ev} The eigenvalues of $\Gamma$ from Lemma \ref{ex:dpg2} are given by
\begin{align*}
 \theta_i           &= \frac{(a_1+1) q^{D-i} - q^i - a_1}{q-1} \qquad \quad (0 \leq i \leq D).
\end{align*}
\end{lemma}

\noindent We bring in some notation. For an integer $n\geq 0$ define
\begin{align*} 
\lbrack n \rbrack_q = \frac{q^n-1}{q-1}.
\end{align*}
\noindent We further define
\begin{align*}
\lbrack n \rbrack^!_q = \lbrack n \rbrack_q \lbrack n-1 \rbrack_q \cdots \lbrack 2 \rbrack_q \lbrack 1 \rbrack_q.
\end{align*}
We interpret $\lbrack 0 \rbrack^!_q=1$. For $0 \leq i \leq n$ define the $q$-binomial coefficient
\begin{align*}
\binom{n}{i}_q = \frac{\lbrack n \rbrack^!_q}{\lbrack i \rbrack^!_q \lbrack n-i \rbrack^!_q}.
\end{align*}

\begin{lemma} \label{lem:ki} {\rm (See \cite[p.~354]{Nbbit}.)} The valencies of $\Gamma$ are
\begin{align*}
k_i = (a_1+1)^i q^{\binom{i}{2}}  \binom{D}{i}_q \qquad \quad (0 \leq i \leq D).
\end{align*}
\noindent In particular,
\begin{align*}
k = (a_1+1) \frac{q^D-1}{q-1}, \qquad \qquad           k_D = (a_1+1)^D q^{\binom{D}{2}}.
\end{align*}
\end{lemma}
\begin{proof} Use \eqref{eq:ki}.
\end{proof}

\section{The irreducible $T$-modules for a dual polar graph }

\noindent We continue to discuss the dual polar graph $\Gamma=(X,\mathcal R)$. Throughout this section, we  fix $x \in X$ and write $T=T(x)$.
Let $W$ denote an irreducible $T$-module. Then $W$ is thin by \cite[Example~6.1(6)]{tSub3}. In this section we discuss the intersection numbers of $W$;  see
 \cite{cerzo,tSub2, tSub3} or \cite[Section~14]{boyd} for the definitions and basic facts
about these parameters.

\begin{lemma} \label{lem:TW} {\rm (See \cite{genQuad}, \cite[p.~200]{tSub3}, \cite[Lemma~17.2]{boyd}.)}
Let $W$ denote an irreducible $T$-module, with endpoint $r$, dual endpoint $t$, and diameter $d$. The intersection numbers of $W$
are described as follows.  For $0 \leq i \leq d$,
\begin{align*}
c_i(W) &= q^t \frac{q^i-1}{q-1}, \\
a_i(W) &= \frac{ (a_1+1)q^{D-d-t+i} -q^{t+i} -a_1}{q-1}, \\
b_i(W)&= (a_1+1) \frac{q^{D-t}-q^{D-d-t+i} }{q-1}.
\end{align*}
\end{lemma}

\noindent Referring to Lemma \ref{lem:TW}, we are mainly interested in $a_i(W)$. In the next result, we clarify the meaning of $a_i(W)$.

\begin{lemma} \label{lem:aiWM} {\rm (See \cite[Lemma~5.9]{cerzo}.)} Let $W$ denote an irreducible $T$-module, with endpoint $r$ and diameter $d$. Then the following holds on $W$:
\begin{align*}
E^*_{r+i}A E^*_{r+i} = a_i (W) E^*_{r+i} \qquad \qquad (0 \leq i \leq d).
\end{align*}
\end{lemma}

\noindent Let $W$ denote an irreducible $T$-module, with endpoint $r$ and diameter $d$.  Our next goal is to compare the intersection number $a_i(W)$ with the intersection number $a_{r+i}$ of $\Gamma$ $(0 \leq i \leq d)$.

\begin{lemma} \label{lem:Comp1}  Let $W$ denote an irreducible $T$-module, with endpoint $r$ and diameter $d$.
Then for $0\leq i \leq d$,
\begin{align}  a_i(W) \leq a_{r+i}.
\label{eq:INEQ}
\end{align}
\end{lemma}
\begin{proof} By Lemma  \ref{lem:aiWM} and the construction,  $a_i(W)$ is an eigenvalue of the subgraph induced on $\Gamma_{r+i}(x)$. This subgraph is regular with valency $a_{r+i}$. The result follows in view of Lemma 
\ref{lem:REG}. 
\end{proof}

\noindent  Shortly, we will consider the case of equality in  \eqref{eq:INEQ}. The following result will help to illuminate this case.

\begin{lemma} \label{lem:AAdiff} Let $W$ denote an irreducible $T$-module, with endpoint $r$, dual endpoint $t$, and diameter $d$.
Then for $0\leq i \leq d$,
\begin{align}
a_{r+i} -a_i(W) = q^{i+D-d-t} \frac{q^{2t-D+d}-1 +       a_1(q^{r+t-D+d}-1)}{q-1}.  \label{eq:Adiff}
\end{align}
Moreover
\begin{align}  
  q^{2t-D+d}-1\geq 0, \qquad \qquad   q^{r+t-D+d}-1\geq 0.
\label{eq:INEQ2}
\end{align}
\end{lemma}
\begin{proof} To get \eqref{eq:Adiff}, use Lemmas \ref{lem:inter}, \ref{lem:TW}. To get \eqref{eq:INEQ2}, use Proposition \ref{lem:Wineq2}(i) and \eqref{eq:disp}
along with the fact that $q$ is a real number greater than $1$.
\end{proof}

\noindent Next we consider the case of equality in \eqref{eq:INEQ}, under the assumption that $\Gamma$ is nonbipartite.

\begin{proposition} \label{lem:4Views} Assume that $\Gamma$ is nonbipartite. Let $W$ denote an irreducible $T$-module, with endpoint $r$, dual endpoint $t$, and diameter $d$.
 Then the following are equivalent:
\begin{enumerate}
\item[\rm (i)] there exists an integer $i$ $(0 \leq i \leq d)$ such that $a_{r+i}=a_i(W)$;
\item[\rm (ii)] $a_{r+i} = a_i(W)$ for $0 \leq i \leq d$;
\item[\rm (iii)] equality holds everywhere in \eqref{eq:INEQ2};
\item[\rm (iv)] $W$ has displacement $0$.
\end{enumerate}
\end{proposition}
\begin{proof} ${\rm (i)} \Leftrightarrow {\rm (ii)} \Leftrightarrow {\rm (iii)}$ By Lemma \ref{lem:AAdiff}. \\
\noindent ${\rm (iii)} \Leftrightarrow {\rm (iv)}$ By  Proposition  \ref{prop:mix} and Definition \ref{def:DISP} and $a_1 \geq 1$.
\end{proof}

\section{The nucleus of a nonbipartite dual polar graph }
\noindent  Throughout this section, we assume that $\Gamma=(X,\mathcal R)$ is a  nonbipartite dual polar graph. We fix $x \in X$ and write $T=T(x)$.
We will give several descriptions of the nucleus $\mathcal N = \mathcal N(x)$. Recall the standard module $V$.

\begin{proposition}\label{prop:DPNUC}  For $0 \leq i \leq D$,
\begin{align*}
 E^*_i \mathcal N = \lbrace v \in E^*_iV \vert E^*_iA E^*_iv = a_i v\rbrace.
 \end{align*}
\end{proposition}
\begin{proof} 
By Lemma \ref{lem:ModODS}, $V$ is an orthogonal direct sum of irreducible $T$-modules:
\begin{align} 
 V = \sum_{W} W \qquad \qquad \hbox{\rm (orthogonal direct sum).} \label{eq:TW}
 \end{align}
 Definition \ref{def:NUC} and Lemma \ref{lem:Interp} yield
 \begin{align} 
 \mathcal N = \sum W \qquad \qquad \hbox{\rm (orthogonal direct sum),} \label{eq:TW2}
 \end{align}
 where the sum is over all the irreducible $T$-modules $W$ from  \eqref{eq:TW} that have displacement 0.
 Apply $E^*_i$ to each side of \eqref{eq:TW2}; this gives
 \begin{align}
E^*_i \mathcal N = \sum E^*_iW \qquad \qquad \hbox{\rm (orthogonal direct sum),} \label{eq:TW3}
 \end{align}
 where the sum is over all the irreducible $T$-modules $W$ from \eqref{eq:TW} that have displacement 0 and $i \in {\rm Supp}(W)$.
 Using Lemma \ref{lem:aiWM}  we obtain
 \begin{align}
 \lbrace v \in E^*_iV \vert E^*_iA E^*_iv = a_i v\rbrace = \sum E^*_iW \qquad  \hbox{\rm (orthogonal direct sum),} \label{eq:TW4}
 \end{align}
 where the sum is over all the irreducible $T$-modules $W$ from \eqref{eq:TW} such that $i \in {\rm Supp}(W)$ and the endpoint $r$ of $W$ satisfies $a_{i-r}(W)=a_i$.
 The sums in \eqref{eq:TW3}, \eqref{eq:TW4} are equal by Proposition \ref{lem:4Views}. The result follows.
\end{proof} 

\noindent To simplify our notation,  we will use the following convention concerning the vertex $x$.
For $0 \leq i \leq D$, by  a {\it connected component of $\Gamma_i(x)$} we mean a connected component of the subgraph induced on $\Gamma_i(x)$.

\begin{lemma} \label{lem:NPart} For $0 \leq i \leq D$ the subspace $E^*_i\mathcal N$ has an orthogonal basis consisting
of the characteristic vectors of the connected components of $\Gamma_i(x)$.
\end{lemma}
\begin{proof}  The subgraph induced on $\Gamma_i(x)$ is regular with valency $a_i$. The result follows in view of Lemma \ref{lem:REG} and Proposition \ref{prop:DPNUC}. 
\end{proof}

\begin{definition}\label{def:sim}  \rm Using the vertex $x$, we define a binary relation $\sim$ on $X$ as follows. For $y, z \in X$ we declare $y \sim z$
whenever both
\begin{enumerate}
\item[\rm (i)] $\partial(x,y) = \partial(x,z)$;
\item[\rm (ii)]  $y,z$ are in the same connected component of $\Gamma_i(x)$, where $i=\partial(x,y)=\partial(x,z)$.
\end{enumerate}
Note that $\sim$ is an equivalence relation.
\end{definition}
\noindent  The next result is meant to clarify Definition \ref{def:sim}.

\begin{lemma} \label{lem:simMeaning} For $0 \leq i \leq D$ the set $\Gamma_i(x)$ is a disjoint union of $\sim$ equivalence classes.
These equivalence classes are the connected components of  $\Gamma_i(x)$.
\end{lemma}
\begin{proof} By Definition \ref{def:sim}.
\end{proof}


\begin{theorem} \label{thm:DPN} The nucleus $\mathcal N$ has an orthogonal basis consisting of the
characteristic vectors of the $\sim$ equivalence classes.
\end{theorem}
\begin{proof} 
 By Lemma \ref{lem:TwoSum},
the sum $\mathcal N=\sum_{i=0}^D E^*_i \mathcal N$ is orthogonal and direct. The result follows in view of
Lemmas \ref{lem:NPart},    \ref{lem:simMeaning}.
\end{proof}

\begin{corollary} \label{cor:Ndim} The following are the same:
\begin{enumerate}
\item[\rm (i)] the dimension of $\mathcal N$;
\item[\rm (ii)]  the number of $\sim$ equivalence classes.
\end{enumerate}
\end{corollary}
\begin{proof} By Theorem \ref{thm:DPN}.
\end{proof}

\begin{corollary} \label{cor:ENdim} For $0 \leq i \leq D$ the following are the same:
\begin{enumerate}
\item[\rm (i)] the dimension of $ E^*_i \mathcal N$;
\item[\rm (ii)]  the number of $\sim$ equivalence classes that are contained in $\Gamma_i(x)$;
\item[\rm (iii)]  the number of connected components of $\Gamma_i(x)$.
\end{enumerate}
\end{corollary}
\begin{proof} By Lemma \ref{lem:simMeaning} and Theorem \ref{thm:DPN}.
\end{proof}

\noindent Recall the multiplicity numbers ${\rm mult}_r$ from Definition  \ref{def:di}.

\begin{corollary} \label{cor:multN} For $0 \leq i \leq D/2$ the following are the same:
\begin{enumerate}
\item[\rm (i)] $\sum_{r=0}^i {\rm mult}_r$;
\item[\rm (ii)]  the number of connected components of  $\Gamma_i(x)$;
\item[\rm (iii)]  the number of connected components of $\Gamma_{D-i}(x)$.
\end{enumerate}
\end{corollary}
\begin{proof} By  Proposition \ref{prop:Six}  and  Corollary \ref{cor:ENdim}.
\end{proof}

\begin{corollary} \label{cor:LastSub} The subgraph induced on $\Gamma_D(x)$ is connected.
\end{corollary}
\begin{proof} By Corollary \ref{cor:multN} and ${\rm mult}_0=1$.
\end{proof}

\noindent In the next section, we will give formulas for ${\rm dim}\, E^*_i\mathcal N$ $(0 \leq i \leq D)$
and 
${\rm mult}_r$ $(0 \leq r \leq D/2)$. We will also describe the partition of $X$ into the $\sim$ equivalence classes.
We will do these things using the concept of a weak-geodetically closed subgraph of $\Gamma$.

\section{The nucleus of a nonbipartite dual polar graph, continued}
\noindent We continue to discuss the nonbipartite dual polar graph $\Gamma=(X,\mathcal R)$. In this section,
we investigate a type of subgraph of $\Gamma$, said to be weak-geodetically closed \cite{weng97, weng98}. We use these subgraphs to
describe the nucleus of $\Gamma$ with respect to any given vertex.

\begin{definition} \label{def:WGC} \rm (See \cite[Section~1]{weng97}.) A subset $\Omega$ of $X$ is called {\it weak-geodetically closed} whenever $\Omega$ is nonempty and
has the following property: for all $y,z\in \Omega$ and $\xi \in X$,
\begin{align*}
\partial(y,\xi) + \partial(\xi,z) \leq \partial(y,z)+1 \qquad \hbox{implies}\qquad \xi \in \Omega.
\end{align*}
\end{definition}

\begin{remark}\rm In some articles,  the condition in Definition \ref{def:WGC} is called {\it strongly closed}, see for example \cite{hiraki, hiraki2}.
\end{remark}

\begin{definition}\label{def:WGCgraph} \rm By a {\it weak-geodetically closed subgraph} of $\Gamma$, we mean the subgraph induced on a weak-geodetically closed subset of $X$.
For notational convenience, this subset will serve as the name for the subgraph.
\end{definition}

\begin{remark}\rm Let $\Omega$ denote a weak-geodetically closed subgraph of $\Gamma$. Then  $\Omega$ is connected.
For $y, z \in \Omega$ the distance $\partial(y,z)$ (in $\Omega$) is equal to the distance $\partial(y,z)$ (in $\Gamma$).
\end{remark}

\begin{lemma} \label{lem:WGC} {\rm (See  \cite[Section~1]{hiraki2},    \cite[Corollary~5.3]{weng98}.)}  Let $\Omega$ denote a weak-geodetically closed subgraph of $\Gamma$, with diameter denoted by $d$.
Then $\Omega$   is distance-regular with valency $c_d+a_d$ and intersection numbers
\begin{align*}
c_i(\Omega) = c_i, \qquad a_i(\Omega) = a_i, \qquad \quad (0 \leq i \leq d).
\end{align*}
\noindent Moreover, $\Omega$ is a dual polar graph. 
\end{lemma}

\begin{lemma} \label{lem:WGCfind} {\rm (See \cite[Corollary 2.2]{weng97}, \cite[Theorem~7.2]{weng98}.)} 
Pick  $y,z \in X$ and write $i=\partial(y,z)$. Then $y,z$ are contained in a unique weak-geodetically closed subgraph of $\Gamma$ that has diameter $i$.
\end{lemma}

\noindent For the rest of this section,
we fix $x \in X$ and write $T=T(x)$. We consider the nucleus $\mathcal N = \mathcal N(x)$. 

\begin{lemma} \label{lem:NS2} For $0 \leq i \leq D$ the following sets {\rm (i)}, {\rm (ii)} are in bijection:
\begin{enumerate}
\item[\rm (i)] the set of weak-geodetically closed subgraphs  of $\Gamma$ that contain $x$ and have diameter $i$;
\item[\rm (ii)] the set of connected components of  $\Gamma_i(x)$.
\end{enumerate} 
A bijection from {\rm (i)} to {\rm (ii)} sends $\Omega \mapsto \Omega \cap \Gamma_i(x)$.
\end{lemma}
\begin{proof} Assume that $i\not=0$; otherwise the result is trivial.
Let $\Omega$ denote a weak-geodetically closed subgraph of $\Gamma$ that contains $x$ and has diameter $i$. We show that $\Omega \cap \Gamma_i(x)$ is a connected component of 
$\Gamma_i(x)$. By Lemma \ref{lem:WGC}, $\Omega$ is a nonbipartite dual polar graph.
Applying Corollary  \ref{cor:LastSub}  to $\Omega$, we see that the subgraph induced on $\Omega \cap \Gamma_i(x)$ is connected. Therefore, there exists a connected component  $\Delta$ of $\Gamma_i(x)$
that contains $\Omega \cap \Gamma_i(x)$. Assume for the moment that this containment is proper. Since $\Delta$ is connected, there exist adjacent vertices $y, \xi \in \Delta$ such that $y \in \Omega\cap \Gamma_i(x)$
and $\xi \not\in \Omega\cap \Gamma_i(x)$. By construction $x, y \in \Omega$. Also by construction,
\begin{align*}
\partial(x,\xi)+ \partial(\xi,y) = i +1=\partial(x,y)+1.
\end{align*}
By these comments and Definition  \ref{def:WGC}, $\xi \in \Omega$. This yields $\xi \in \Omega \cap \Delta \subseteq \Omega\cap \Gamma_i(x)$, for a contradiction. 
We have shown that $\Omega \cap \Gamma_i(x)$ is equal to $\Delta$ and is therefore  a connected component of 
$\Gamma_i(x)$. By our discussion so far, there exists a function from (i) to (ii) that sends $\Omega \mapsto \Omega \cap \Gamma_i(x)$. This function is  bijective in view of Lemma \ref{lem:WGCfind}.
\end{proof}

\begin{lemma} \label{lem:cc} The following hold for  $0 \leq i \leq D$.
\begin{enumerate}
\item[\rm (i)]
 Each connected component of $\Gamma_i(x)$ has cardinality $(a_1+1)^i q^{\binom{i}{2}}$.
 \item[\rm (ii)]
The number of connected components in $\Gamma_i(x)$ is equal to $\binom{D}{i}_q$.
\end{enumerate}
\end{lemma}
\begin{proof} (i) By Lemma  \ref{lem:NS2},  each connected component of $\Gamma_i(x)$ has the form $\Omega \cap \Gamma_i(x)$, where $\Omega$
is a weak-geodetically closed subgraph of $\Gamma$ that contains $x$ and has diameter $i$. By Lemma  \ref{lem:WGC}, $\Omega$ is a dual polar graph.
By Lemma \ref{lem:ki} (applied to $\Omega$) and  Lemma  \ref{lem:WGC},
\begin{align*} 
\vert \Omega \cap \Gamma_i(x) \vert = k_i(\Omega) = (a_1+1)^i q^{\binom{i}{2}}.
\end{align*}
\noindent (ii) The desired number is equal to $\vert \Gamma_i(x) \vert$ divided by the common cardinality from (i). By Lemma \ref{lem:ki},
\begin{align*}
\vert \Gamma_i(x) \vert = k_i = (a_1+1)^i q^{\binom{i}{2}} \binom{D}{i}_q.
\end{align*}
By these comments, the desired number is equal to $\binom{D}{i}_q$.
\end{proof}

\begin{proposition} \label{prop:END} We have
\begin{align*}
{\rm dim}\, E^*_i\mathcal N = \binom{D}{i}_q \qquad \qquad (0 \leq i \leq D).
\end{align*}
\end{proposition}
\begin{proof} By Corollary \ref{cor:ENdim} and Lemma  \ref{lem:cc}(ii).
\end{proof}

\begin{corollary} \label{cor:Ndimf}
We have
\begin{align*}
{\rm dim}\, \mathcal N = \sum_{i=0}^D \binom{D}{i}_q.
\end{align*}
\end{corollary}
\begin{proof} By Lemma \ref{lem:TwoSum} and Proposition  \ref{prop:END}.
\end{proof}

\noindent Recall the multiplicity numbers ${\rm mult}_r$ from Definition  \ref{def:di}. Recall that ${\rm mult}_0=1$. 

\begin{proposition} We have
\begin{align*}
{\rm mult}_r = \binom{D}{r}_q - \binom{D}{r-1}_q \qquad \quad (1 \leq r \leq D/2).
\end{align*}
\end{proposition}
\begin{proof} By Propositions  \ref{prop:Six}, \ref{prop:END} we obtain
\begin{align*}
\sum_{r=0}^i {\rm mult}_r = \binom{D}{i}_q \qquad \qquad (0 \leq i \leq D/2).
\end{align*}
The result follows.
\end{proof}

\noindent  Recall  the equivalence relation $\sim$ from Definition \ref{def:sim}.
\begin{lemma} \label{lem:equit} Consider the partition of $X$ into $\sim$ equivalence classes. This partition is equitable in the sense of \cite[Section~A.4]{Nbcn}.
\end{lemma}
\begin{proof} The nucleus $\mathcal N$ is a $T$-module, so $A \mathcal N \subseteq \mathcal N$. The result follows in view
of Theorem  \ref{thm:DPN}.
\end{proof}
\noindent  In Lemma \ref{lem:equit} we obtained an equitable partition of $X$. Our next goal is to compute the corresponding structure constants. This will be
done over the next two lemmas.

\begin{lemma} \label{lem:equit1} Let $0 \leq i \leq D$, and let $\Delta$ denote a $\sim$ equivalence class contained in $\Gamma_i(x)$.
Then each vertex in $\Delta$ is adjacent to exactly $a_i$ vertices in $\Delta$.
\end{lemma}
\begin{proof} The subgraph $\Gamma_i(x)$ is regular with valency $a_i$, and $\Delta$ is a connected component of $\Gamma_i(x)$. 
\end{proof}

\begin{lemma} \label{lem:equit2} Let $1 \leq i \leq D$, and let $\Delta^{\pm}$  denote $\sim$ equivalence classes such that $\Delta^- \subseteq \Gamma_{i-1}(x)$
and $\Delta^+ \subseteq \Gamma_i(x)$. Assume that there exists at least one edge between $\Delta^-$ and $\Delta^+$.
 Then:
\begin{enumerate} 
\item[\rm (i)]  each vertex in $\Delta^+$ is adjacent to exactly $1$ vertex in $\Delta^-$;
\item[\rm (ii)] each vertex in $\Delta^-$ is adjacent to exactly $(a_1+1)q^{i-1}$ vertices in $\Delta^+$.
\end{enumerate}
\end{lemma}
\begin{proof} By Lemma \ref{lem:equit}, there exist nonnegative integers $m,n$ such that
\begin{enumerate} 
\item[$\bullet$]  each vertex in $\Delta^+$ is adjacent to exactly $m$ vertex in $\Delta^-$;
\item[$\bullet$] each vertex in $\Delta^-$ is adjacent to exactly $n$ vertices in $\Delta^+$.
\end{enumerate}
\noindent  By assumption, there exists at least one edge between $\Delta^-$ and $\Delta^+$. Counting the number of such edges in two ways, we obtain
\begin{align*}
\vert \Delta^- \vert n = \vert \Delta^+ \vert m.
\end{align*}
By Lemmas  \ref{lem:simMeaning},  \ref{lem:cc}(i) we have
\begin{align*}
\vert \Delta^- \vert = (a_1 +1)^{i-1} q^{\binom{i-1}{2}}, \qquad \qquad \vert \Delta^+ \vert = (a_1 +1)^i q^{\binom{i}{2}}.
\end{align*}
By these comments, 
\begin{align*}
n= (a_1+1)q^{i-1}m.
\end{align*}
 It remains to show that $m=1$. By construction, $m\geq 1$. We assume that $m\geq 2$,
and get a contradiction. Pick $y \in \Delta^+$. By assumption, there exist distinct vertices $z_1, z_2 \in \Delta^-\cap \Gamma(y)$. Observe that
 $\partial(z_1, z_2) \in \lbrace 1,2\rbrace$.
By Lemmas \ref{lem:simMeaning}, \ref{lem:NS2} there exists a weak-geodetically closed subgraph $\Omega$ of $\Gamma$ that contains $x$ and has diameter $i-1$, such that
$\Delta^- = \Omega \cap \Gamma_{i-1}(x)$. We have $z_1, z_2 \in \Delta^-\subseteq \Omega$. We have $\partial(z_1,y)+\partial(y,z_2) = 1+1\leq \partial(z_1,z_2)+1$,
so $y \in \Omega$ in view of Definition \ref{def:WGC}. We have $x,y \in \Omega$ and $\partial(x,y)=i$, contradicting the assumption that $\Omega$ has diameter $i-1$.
We have shown that $m=1$, and the result follows.
\end{proof}

\section{The projective geometry $L_D(q)$}

We continue to discuss the nonbipartite dual polar graph $\Gamma=(X,\mathcal R)$. Throughout this section, we fix
$x \in X$ and write $T=T(x)$. In this section, we explain how the nucleus $\mathcal N = \mathcal N(x)$
is related to the projective geometry $L_D(q)$. 
\medskip

\noindent
We now  define $L_D(q)$. We will work with the finite field 
$GF(q)$ associated with $\Gamma$ from Example \ref{ex:dp}.

\begin{definition}\label{def:LDq} \rm Let $\bf V$ denote a vector space over $GF(q)$ that has dimension $D$.
 Let the set $\mathcal P$ consist of the subspaces of $\bf V$. 
 Define a partial order $\leq$ on $\mathcal P$ such that for $\eta, \zeta \in \mathcal P$,
$\eta \leq \zeta$ whenever $\eta \subseteq \zeta$.
The poset $\mathcal P, \leq $ is denoted $L_D(q)$ and called a {\it projective geometry}.  
\end{definition}
\noindent For basic 
 information about $L_D(q)$, see \cite[Section~9.3]{Nbcn},  \cite[Chapter~1]{Ncameron},
 \cite{murali, murali2}, 
  \cite[Example 3.1(5) with $M=N$]{uniform},  \cite[Section~7]{LSintro}, \cite{Lnq, LNq2}.
\medskip

\noindent By Example \ref{ex:dp}  and the construction, the vertex $x$ is a vector space over $GF(q)$ that has dimension $D$. 
For notational convenience, we always take the ${\bf V} = x$.

\begin{definition}\label{def:PP} \rm 
 For $\eta, \zeta \in \mathcal P$, we say  that {\it $\zeta$ covers $\eta$} whenever $\eta \subseteq \zeta$ and ${\rm dim}\, \zeta - {\rm dim}\, \eta = 1$. 
We say that $\eta, \zeta$ are {\it adjacent} whenever one of $\eta, \zeta$ covers the other one. The set $\mathcal P$  together with the adjacency
relation, forms an undirected graph.
For $\eta \in \mathcal P$, let the set $\mathcal P(\eta)$ consist
of the elements in $\mathcal P$ that are adjacent to $\eta$. For $0 \leq i \leq D$, let the set $\mathcal P_i$ consist of the elements in $\mathcal P$
that have dimension $D-i$. Note that $\mathcal P_0 = \lbrace x \rbrace$. For notational convenience, define $\mathcal P_{-1}=\emptyset$ and
$\mathcal P_{D+1}=\emptyset$.
\end{definition}

\noindent In the next two lemmas, we describe some basic combinatorial features of $\mathcal P$.

\begin{lemma} \label{lem:YYi} For $0 \leq i \leq D$, 
each vertex in $\mathcal P_i$ is adjacent to exactly  $\lbrack i \rbrack_q$ vertices in $\mathcal P_{i-1}$ and exactly $\lbrack D-i \rbrack_q$ vertices in $\mathcal P_{i+1}$.
\end{lemma}
\begin{proof} By routine combinatorial counting; see for example \cite[Section~9.3]{Nbcn}.
\end{proof}

\noindent The following result is well known; see for example \cite[Lemma~9.3.2]{Nbcn}.

\begin{lemma} \label{lem:SZ} {\rm (See \cite[Lemma~9.3.2]{Nbcn}.)}  We have
\begin{align*}
\vert \mathcal P_i \vert =  \binom{D}{i}_q \qquad \quad (0 \leq i \leq D).
\end{align*}
\end{lemma}

\noindent  We will use $\mathcal P$ to obtain a basis for $\mathcal N$.
\medskip

\begin{definition}\label{def:etaV} \rm For $\eta \in \mathcal P$ let  $\eta^\vee$ denote the subgraph of $\Gamma$ consisting of the vertices in $X$ that contain $\eta$.
\end{definition}

\begin{lemma} \label{lem:etaV} The following hold for  $\eta \in \mathcal P$:
\begin{enumerate}
\item[\rm (i)] $\eta^\vee$ is a weak-geodetically closed subgraph of $\Gamma$ that contains $x$ and has diameter $i = D-{\rm dim}\,\eta$;
\item[\rm (ii)] $\eta^\vee \cap \Gamma_i(x) = \lbrace \xi \in X \vert \eta =x \cap \xi \rbrace$.
\end{enumerate}
\end{lemma}
\begin{proof} (i) This is a routine consequence of the theory discussed in Remark \ref{rem:U}. \\
\noindent (ii) First we check the inclusion $\subseteq$. Pick $y \in \eta^\vee \cap \Gamma_i(x)$. We have $y \in X$ and $\eta \subseteq y$ and $\partial(x,y)=i$.
We have $\eta \subseteq x$ since $\eta \in \mathcal P$, so $\eta \subseteq x \cap y$. Both $\eta$ and $x \cap y$ have dimension $D-i$, 
so $\eta = x \cap y$.  \\
\noindent Next we check the inclusion $\supseteq$. Pick $y \in X$ such that $\eta = x \cap y$. We have $\eta \subseteq y$, so $y \in \eta^\vee$.
We have ${\rm dim}\, x \cap y = {\rm dim}\,\eta = D-i$, so $\partial (x,y) = i$. By these comments, $y \in \eta^\vee \cap \Gamma_i(x)$.
\end{proof}

\begin{lemma} \label{lem:WGCsame} Let $0 \leq i \leq D$ and $y \in \Gamma_i(x)$. Then the  following are the same:
\begin{enumerate}
\item[\rm (i)] the unique weak-geodetically closed subgraph of $\Gamma$ that has diameter $i$ and contains $x$ and $y$;
\item[\rm (ii)]  the subgraph $\eta^\vee$, where $\eta = x \cap y$.
\end{enumerate}
\end{lemma}
\begin{proof}  It suffices to show that $\eta^\vee$ is a weak-geodetically closed subgraph of $\Gamma$ that has diameter $i$ and contains $x$ and $y$.
We have 
\begin{align*}
{\rm dim}\, \eta = {\rm dim}\, x \cap y = D - \partial(x,y) = D-i.
\end{align*}
By Lemma \ref{lem:etaV}(i), $\eta^\vee$ is a weak-geodetically closed subgraph of $\Gamma$ that contains $x$ and has diameter $i$. By construction,
$\eta \in y$ so $y \in \eta^\vee$.
\end{proof}

\begin{lemma} \label{lem:simE} For $y, z \in X$ the following are equivalent:
\begin{enumerate}
\item[\rm (i)] $y\sim z$;
\item[\rm (ii)] $x \cap y = x \cap z$.
\end{enumerate}
\end{lemma}
\begin{proof} ${\rm (i)} \Rightarrow {\rm (ii)}$ By Definition \ref{def:sim}, there exists an integer $i$ $(0 \leq i \leq D)$ and a 
 connected component $\Delta$ of $\Gamma_i(x)$, such that $y,z \in \Delta$.
 By Lemma \ref{lem:NS2}, there exists a unique weak-geodetically closed subgraph $\Omega$ of $\Gamma$ that has diameter $i$ and contains $x$
 such that $\Delta = \Omega \cap \Gamma_i(x)$. By construction $y,z \in \Omega\cap \Gamma_i(x)$. 
 By Lemma \ref{lem:WGCsame}, 
 \begin{align*}
 (x \cap y)^\vee = \Omega = (x \cap z)^\vee.
 \end{align*}
 We may now argue using Lemma \ref{lem:etaV}, 
 \begin{align*}
 y \in \Omega \cap \Gamma_i(x) = (x \cap z)^\vee \cap \Gamma_i(x) = \lbrace \xi \in X \vert x \cap z = x \cap \xi \rbrace.
 \end{align*}
Therefore $x \cap y = x \cap z$. \\
\noindent  ${\rm (ii)} \Rightarrow {\rm (i)}$ Define $\eta = x \cap y = x \cap z$. We have
\begin{align*}
\partial(x,y)= D - {\rm dim}\, \eta = \partial(x,z).
\end{align*}
Define $i = \partial(x,y)=\partial(x,z)$ and $\Omega = \eta^\vee$.
By Lemma \ref{lem:etaV}(i), $\Omega$ is a weak-geodetically closed subgraph of $\Gamma$ that contains $x$ and has diameter $i$.
By construction
$y, z \in \Omega \cap \Gamma_i(x)$.
By Lemma \ref{lem:NS2}, the vertices $y,z$ are  in the same connected component of $\Gamma_i(x)$.  Now $y \sim z$ by Definition \ref{def:sim}.
\end{proof}

\noindent Recall the standard module $V$ of $\Gamma$. 

\begin{definition}\label{def:etaN} \rm For $\eta \in \mathcal P$ we define a vector $\eta^\mathcal N \in V$ as follows:
\begin{align*}
\eta^\mathcal N = \sum_{\stackrel{y \in X}{x \cap y=\eta}} \hat y.
\end{align*}
\end{definition}
\noindent By Lemma \ref{lem:simE} and the construction, the vector $\eta^\mathcal N$ 
in Definition \ref{def:etaN} is the characteristic vector of the $\sim$ equivalence class that corresponds to the weak-geodetically closed
subgraph $\eta^\vee$, via the bijection in Lemma \ref{lem:NS2}.

\begin{theorem}\rm \label{thm:bij} We give a bijection from $\mathcal P$ to the 
basis for $\mathcal N$ in Theorem \ref{thm:DPN}. The bijection sends $\eta \to \eta^{\mathcal N}$ for all $\eta \in \mathcal P$.
\end{theorem}
\begin{proof}  The basis for  $\mathcal N$ in  Theorem \ref{thm:DPN}  consists of the characteristic vectors of the $\sim $ equivalence classes.
The set $\lbrace \eta^\mathcal N \vert \eta \in \mathcal P\rbrace$ consists of 
these characteristic vectors, by the comments below
Definition \ref{def:etaN}. The result follows.
\end{proof}

\noindent We now bring in the adjacency matrix $A$ of $\Gamma$, and the dual adjacency matrix $A^*=A^*(x)$ of $\Gamma$ with respect to $x$.

\begin{theorem}\label{thm:Pmain} We give the action of $A, A^*$  on the basis 
 $\lbrace \eta^\mathcal N \vert \eta \in \mathcal P\rbrace$ for $\mathcal N$.
For $0 \leq i \leq D$ and  $\eta \in \mathcal P_i$ we have
\begin{align*} 
A \eta^\mathcal N &= a_1 \frac{q^i-1}{q-1} \eta^\mathcal N + \sum_{\zeta \in \mathcal P(\eta) \cap \mathcal P_{i+1}} \zeta^\mathcal N +
(a_1+1)q^{i-1} \sum_{\zeta \in \mathcal P(\eta) \cap \mathcal P_{i-1}} \zeta^\mathcal N; \\
A^* \eta^\mathcal N &= \theta^*_i \eta^\mathcal N. 
\end{align*}
\end{theorem}
\begin{proof} 
The $A$ action 
follows from  Lemmas \ref{lem:equit1}, \ref{lem:equit2}. The $A^*$ action  
follows from
\begin{align*}
\eta^{\mathcal N} \in 
{\rm Span} \lbrace {\hat y} \vert y \in X, \;  x \cap y=\eta \rbrace 
\subseteq {\rm Span} \lbrace {\hat y} \vert y \in \Gamma_i(x) \rbrace =          E^*_iV.
\end{align*}
\end{proof}

\noindent We finish this section with some remarks.  By Theorem \ref{thm:Pmain},  the action of $A$ on $\mathcal N$ becomes a weighted adjacency map for $L_D(q)$.
This weighted adjacency map is close to the one in \cite[Definition~7.1]{LNq2}.
Indeed,  $A + \frac{a_1}{q-1}I$ corresponds to the weighted adjacency map  in \cite[Definition~7.1]{LNq2}, where the parameter $\varphi $ from
\cite[Definition~7.1]{LNq2} satisfies $\varphi=a_1+1$. 
\medskip

\noindent We just mentioned some weighted adjacency maps for $L_D(q)$.
A similar weighted adjacency map for $L_D(q)$ showed up earlier in the work of 
Bernard, Cramp{\'e}, and Vinet
\cite[Theorem~7.1]{PAB} concerning the symplectic dual polar graph $C_D(q)$ with $q$ a prime  (see Example \ref{ex:dp}).
The approach in \cite{PAB} is quite different from ours.

\section{Directions for future research}

In this section, we give some open problems concerning the nucleus.
\medskip

\noindent
 Throughout this section,
let $\Gamma=(X, \mathcal R)$ denote a $Q$-polynomial distance-regular graph. Recall the standard module $V$. Fix $x \in X$ and write $T=T(x)$.
Recall the nucleus $\mathcal N=\mathcal N(x)$.

\begin{problem}\rm Compute $\mathcal N$ under the assumption that $\Gamma$ belongs to a known infinite family of $Q$-polynomial distance-regular
graphs with unbounded diameter; see \cite[Section~6.4]{Nbbit}, \cite[Section~3.6]{Nbannai}, \cite[Section~8]{Nbcn}.
\end{problem}

\begin{problem} \rm Determine how the algebraic structure of $\mathcal N$ depends on the choice of $x$.
\end{problem}

\begin{problem} \rm A vector in $V$ is called $0/1$ whenever every entry  is $0$ or $1$.
For the examples of $\Gamma$  in Sections 8, 9 the subspace  $\mathcal N$ has an orthogonal basis consisting of some $0/1$ vectors in $V$.
 Perhaps $\mathcal N$ has this type of basis in all cases.
It would be interesting to explore this issue.
\end{problem}

\begin{problem} \rm By \cite[Lemma~5.4]{tSub3}, the adjacency matrix $A$ and dual adjacency matrix $A^*=A^*(x)$ satisfy a pair of relations called the tridiagonal relations.
Perhaps the restrictions of $A, A^*$ to $\mathcal N$ satisfy some additional relations. 
It would be interesting to explore this issue.
\end{problem}


\begin{problem}\rm For a nonempty subset $C \subseteq X$, consider the characteristic vector $\widehat C = \sum_{y \in C} {\hat y}$. Determine those $C$ for which
$\widehat C \in \mathcal N(y)$ for all $y \in C$. This happens if $C$ is a descendant in the sense of \cite[Section~2]{tanaka2}. The articles \cite{wdw, tanaka1} might be relevant for
this problem.
\end{problem}

\begin{problem} \rm Determine those  $\Gamma$ for which $\cap_{y \in X} \mathcal N(y)$ is nonzero. This happens
if  $\Gamma$ is a bipartite antipodal 2-cover (see Examples \ref{ex:cube}, \ref{ex:2Hom}). In this case $\cap_{y \in X} \mathcal N(y)=V$.
\end{problem}

\section{Acknowledgement} 
\noindent 
The author thanks Pierre-Antoine Bernard for a discussion about how a weighted adjacency
matrix for $L_D(q)$ shows up in the work \cite{PAB} concerning the symplectic dual polar graphs.
The author thanks Hajime Tanaka for some discussions about the descendants of a $Q$-polynomial distance-regular graph.
 These conversations helped to motivate the present paper.
The author thanks Pierre-Antoine Bernard, Allen Herman, Kazumasa Nomura, and Arjana  \v Zitnik for reading the manuscript carefully, and sending valuable comments.


\bigskip


\noindent Paul Terwilliger \hfil\break
\noindent Department of Mathematics \hfil\break
\noindent University of Wisconsin \hfil\break
\noindent 480 Lincoln Drive \hfil\break
\noindent Madison, WI 53706-1388 USA \hfil\break
\noindent email: {\tt terwilli@math.wisc.edu }\hfil\break

\section{Statements and Declarations}

\noindent {\bf Funding}: The author declares that no funds, grants, or other support were received during the preparation of this manuscript.
\medskip

\noindent  {\bf Competing interests}:  The author  has no relevant financial or non-financial interests to disclose.
\medskip

\noindent {\bf Data availability}: All data generated or analyzed during this study are included in this published article.


\begin{thebibliography}{10}

\bibitem{artin}
E.~Artin. 
\newblock {\em Geometric Algebra}.
\newblock Interscience, New York, 1957.



\bibitem{Nbbit}
E.~Bannai, Et.~Bannai, T.~Ito, R.~Tanaka.
\newblock
{\em Algebraic Combinatorics.}
\newblock De Gruyter Series in Discrete Math and Applications 5.
De Gruyter, 2021.      \\
https://doi.org/10.1515/9783110630251
 

\bibitem{Nbannai}
E.~Bannai and T.~Ito.
\newblock
{\em Algebraic Combinatorics, I. Association schemes.}
\newblock Benjamin/Cummings, Menlo Park, CA, 1984.


\bibitem{PAB}
P.~Bernard, N.~Cramp{\'e}, L.~Vinet.
\newblock
The Terwilliger algebra of symplectic dual polar graphs, the subspace lattices and $U_q(\mathfrak{sl}_2)$.
\newblock{\em  Discrete Math.} 345 (2022) Paper 113169, 19 pp.;
{\tt arXiv:2108.13819}.

\bibitem{biggs}
N.~L.~Biggs.
\newblock
{\em Algebraic Graph Theory}.
\newblock
Cambridge Tracts in Math., No. 67
Cambridge University Press, London, 1974.

 
 
\bibitem{Nbcn}
A.~E.~Brouwer,  A.~Cohen, A.~Neumaier.
\newblock{\em Distance Regular-Graphs.}
\newblock Springer-Verlag, Berlin, 1989.

\bibitem{wdw}
A.~Brouwer,  C.~Godsil, J.~Koolen, W.~Martin.
\newblock
Width and dual width of subsets in polynomial association schemes.
\newblock{\em 
J. Combin. Theory Ser.}  A 102 (2003) 255--271.
   

\bibitem{Ncameron}
P.~J.~Cameron.
\newblock  Projective and polar spaces.
\newblock QMW Maths Notes, 13. Queen Mary and Westfield College, School of Mathematical Sciences, London, 1992.


\bibitem{caughman}
J.~S.~Caughman.
\newblock
 The Terwilliger algebras of bipartite $P$- and $Q$-polynomial schemes.
 \newblock{\em 
Discrete Math.} 196 (1999)  65--95.	

\bibitem{cerzo}
D.~Cerzo.
\newblock
Structure of thin irreducible modules of a $Q$-polynomial distance-regular graph.
\newblock{\em 
Linear Algebra Appl.} 433 (2010) 1573--1613; {\tt arXiv:1003.5368}.



\bibitem{curtin2H}
B.~Curtin.
\newblock
 2-homogeneous bipartite distance-regular graphs.
 \newblock{\em 
Discrete Math.} 187 (1998) 39--70.

\bibitem{curtin}
B.~Curtin.
\newblock
 Bipartite distance-regular graphs I.
 \newblock{\em 
Graphs Combin.} 15 (1999)  143--158.



\bibitem{curtin2HT}
B.~Curtin.
\newblock
The Terwilliger algebra of a 2-homogeneous bipartite distance-regular graph.
\newblock{\em 
J. Combin. Theory Ser. B} 81 (2001) 125--141.





\bibitem{Ndkt}   
E.~R.~van Dam, J.~H.~Koolen, H.~Tanaka.
\newblock Distance-regular graphs.
\newblock{\em Electron. J. Combin.} (2016) DS22;
{\tt arXiv:1410.6294}.



\bibitem{hiraki}
A.~Hiraki.
\newblock
Strongly closed subgraphs in a distance-regular graph with $c_2>1$.
\newblock{\em 
Graphs Combin.} 24 (2008) 537--550.

\bibitem{hiraki2}
A.~Hiraki.
\newblock
 A characterization of the Hamming graphs and the dual polar graphs by completely regular subgraphs.
 \newblock{\em 
Graphs Combin.} 28 (2012) 449--467.





\bibitem{murali}
S.~Ghosh and M.~Srinivasan.
\newblock  
A $q$-analog of the adjacency matrix of the $n$-cube.
\newblock{\em 
Algebr. Comb.} 6 (2023)  707--725;
{\tt arXiv:2204.05540}.

\bibitem{go}
J.~T.~Go.
\newblock
The Terwilliger algebra of the hypercube.
\newblock{\em 
European J. Combin.} 23 (2002)  399--429.




\bibitem{itoOq}
T.~ Ito.
\newblock TD-pairs and the $q$-Onsager algebra.
\newblock{\em 
 Sugaku Expositions}  32  (2019)  205--232.
	
	
		



\bibitem{IKT}
T.~Ito, K.~Nomura,  P.~Terwilliger.
\newblock A classification of sharp tridiagonal pairs.
\newblock{\em  Linear Algebra Appl. } 435  (2011)  1857--1884; {\tt arXiv:1001.1812}.



\bibitem{NsomeAlg}
T.~Ito, K.~Tanabe, P.~Terwilliger.
\newblock Some algebra related to $P$- and $Q$-polynomial association schemes.
\newblock{\em
Codes and Association Schemes (Piscataway NJ, 1999), 167--192, DIMACS
Ser. Discrete Math. Theoret. Comput. Sci.}
{56}, Amer. Math. Soc., Providence RI 2001;
{\tt arXiv:math.CO/0406556}.


\bibitem{tdanduq}
T.~Ito and P.~Terwilliger.
\newblock {Tridiagonal pairs and the quantum affine 
algebra
$U_q({\widehat{sl}}_2)$.}
\newblock {\em Ramanujan J.}
{13} (2007) 39--62;
{\tt arXiv:math.QA/0310042}.

 
 \bibitem{TDqRac}
 T.~Ito and P.~Terwilliger.
 \newblock
 Tridiagonal pairs of $q$-Racah type.
 \newblock{\em 
J. Algebra} 322 (2009)  68--93; {\tt arXiv:0807.0271}.
 
 \bibitem{augIto}
T.~Ito and P.~Terwilliger.
\newblock
The augmented tridiagonal algebra.
\newblock{\em  Kyushu J. Math.} 64 (2010) 81--144; {\tt arXiv:0904.2889}.
 
 







\bibitem{kim2}
J.~Kim.
\newblock
A duality between pairs of split decompositions for a $Q$-polynomial distance-regular graph
\newblock{\em 
Discrete Math.} 310 (2010) 1828--1834; {\tt arXiv:0705.0167}.

\bibitem{genQuad}
F.~Levstein and C.~Maldonado.
\newblock
Generalized quadrangles and subconstituent algebra.
\newblock{\em 
Cubo } 12 (2010)  53--75.





\bibitem{mamart}
S.~Mamart.
\newblock
A group commutator involving the last distance matrix and dual distance matrix of a $Q$-polynomial distance-regular graph: the Hamming graph case.
\newblock{\em 
Graphs Combin.} 34 (2018) 803--817; {\tt arXiv:1801.05494}.





\bibitem{nomSplit}
  K.~ Nomura and P.~Terwilliger.
  \newblock
  The split decomposition of a tridiagonal pair.
  \newblock{\em
Linear Algebra Appl.} 424 (2007) 339--345; {\tt arXiv:math/0612460}.
 
 
 \bibitem{switch}
 K.~ Nomura and P.~Terwilliger.
 \newblock
 The switching element for a Leonard pair.
 \newblock{\em 
Linear Algebra Appl.} 428 (2008) 1083--1108; {\tt arXiv:math/0608623}.
 
 
\bibitem{sharp} 
   K.~ Nomura and P.~Terwilliger.
   \newblock
 Sharp tridiagonal pairs.
 \newblock{\em 
Linear Algebra Appl. } 429 (2008) 79--99; {\tt arXiv:0712.3665}.

\bibitem{structure}
    K.~ Nomura and P.~Terwilliger.
    \newblock
 The structure of a tridiagonal pair.
 \newblock{\em 
Linear Algebra Appl.} 429 (2008) 1647--1662; {\tt arXiv:0802.1096}.


 \bibitem{pasc}
 A.~Pascasio.
 \newblock  On the multiplicities of the primitive idempotents of a $Q$-polynomial distance-regular graph.
\newblock{\em European J. Combin.}  23  (2002) 1073--1078.


\bibitem{stanton}
D.~Stanton.
\newblock
Harmonics on posets.
\newblock{\em 
J. Combin. Theory Ser. A} 40 (1985) 136--149.
 
\bibitem{murali2}
 M.~Srinivasan.
 \newblock  A positive combinatorial formula for the complexity of the $q$-analog of the $n$-cube.
\newblock{\em  Electron. J. Combin. } 19  (2012)  Paper 34, 14 pp.; {\tt arXiv:1111.1799}.

\bibitem{tanaka1}
H.~Tanaka.
\newblock
Classification of subsets with minimal width and dual width in Grassmann, bilinear forms and dual polar graphs.
\newblock{\em 
J. Combin. Theory Ser. A } 113 (2006) 903--910.

\bibitem{tanaka2}
H.~Tanaka.
\newblock
Vertex subsets with minimal width and dual width in $Q$-polynomial distance-regular graphs.
\newblock{\em 
Electron. J. Combin. } 18 (2011) Paper 167, 32 pp.;   {\tt arXiv:1011.2000}.


 
 
 
\bibitem{uniform}
 P.~Terwilliger.
 \newblock The incidence algebra of a uniform poset.
\newblock Coding theory and design theory, Part I, 
 193--212, IMA Vol. Math. Appl., 20, Springer, New York,  1990.
 
 
 
\bibitem{tSub1} 
P.~Terwilliger. 
\newblock The subconstituent algebra of
an association scheme I.
\newblock{\em
J. Algebraic Combin.}
{ 1} (1992), 363--388.  

\bibitem{tSub2} 
P.~Terwilliger. 
\newblock The subconstituent algebra of
an association scheme II.
\newblock{\em
J. Algebraic Combin.}
{ 2} (1993), 73--103.  

\bibitem{tSub3} 
P.~Terwilliger. 
\newblock The subconstituent algebra of
an association scheme III. 
\newblock{ \em
J. Algebraic Combin.}
{ 2} (1993), 177--210.  

\bibitem{qm}
P.~Terwilliger.
\newblock
Quantum matroids, in: E. Bannai and A. Munemasa (Eds.) Progress in Algebraic Combinatorics.
\newblock{\em
Adv. Stud. Pure Math.} 24.
\newblock Mathematical Society of Japan, Tokyo. (1996) 323--441.
 
\bibitem{LS99}
P.~Terwilliger.
\newblock Two linear transformations each tridiagonal with respect to an
  eigenbasis of the other.
  \newblock {\em Linear Algebra Appl.}  330 (2001) 149--203;
{\tt arXiv:math.RA/0406555}.
 

 

\bibitem{LSintro}
P.~Terwilliger.
\newblock
Introduction to Leonard pairs.
\newblock{\em 
J. Comput. Appl. Math.} 153 (2003)  463--475.
	
						 
\bibitem{ds}
P.~Terwilliger.
\newblock The displacement and split decompositions
for a $Q$-polynomial distance-regular graph.
\newblock{\em Graphs Combin.} { 21} (2005) 263--276.
{\tt arXiv:math.CO/0306142}.




\bibitem{LSnotes}
P.~Terwilliger.
\newblock
Notes on the Leonard system classification.
\newblock{\em 
 Graphs Combin. } 37  (2021) 1687--1748;
{\tt arXiv:2003.09668}.

\bibitem{Nint}
P.~Terwilliger.
\newblock
Distance-regular graphs, the subconstituent algebra, 
and the $Q$-polynomial property.
\newblock
London Math. Soc. Lecture Note Ser., 487
Cambridge University Press, London, 2024, 430--491; {\tt arXiv:2207.07747}.
 
\bibitem{altDRG}
P.~Terwilliger.
\newblock 
 Tridiagonal pairs, alternating elements, and distance-regular graphs.
 \newblock{\em 
J. Combin. Theory Ser.} A 196 (2023) Paper No. 105724, 41 pp.; {\tt arXiv:2207.07741}.
 
 
 
\bibitem{Lnq}
P.~Terwilliger.
\newblock  A $Q$-polynomial structure associated with the projective geometry $L_N(q)$.
\newblock{\em  Graphs Combin.}  39  (2023)   no. 4, 63; {\tt arXiv:2208.13098}.

\bibitem{LNq2}
P.~Terwilliger.
\newblock Projective geometries,  $Q$-polynomial structures,
and quantum groups.
\newblock Preprint; {\tt arXiv:2407.14964}. 

\bibitem{vidunas}
P.~Terwilliger and R.~Vidunas.
\newblock Leonard pairs and the Askey-Wilson relations.
 \newblock {\em J. Algebra Appl. } 3  (2004) 411--426; {\tt  arXiv:math/0305356}.
 
 
 \bibitem{weng97}
 C.W.~Weng.
 \newblock
 $D$-bounded distance-regular graphs.
 \newblock{\em 
European J. Combin.} 18 (1997) 211--229.
 
 
 
 \bibitem{weng98}
 C.W.~Weng.
 \newblock
 Weak-geodetically closed subgraphs in distance-regular graphs.
 \newblock{\em 
Graphs Combin.} 14 (1998)  275--304.
 
 

 \bibitem{boyd}
 C.~ Worawannotai.
 \newblock Dual polar graphs, the quantum algebra $U_q(\mathfrak{sl}_2)$, and Leonard systems of dual $q$-Krawtchouk type.
 \newblock {\em Linear Algebra Appl. } 438  (2013) 443--497;     {\tt arXiv:1205.2144}.

			

\end{thebibliography}
\end{document}